\documentclass[oneside, english, reqno, a4paper, 10pt]{amsart}
\usepackage[utf8]{inputenc}
\usepackage{babel}
\usepackage[foot]{amsaddr}	
\usepackage{amsmath}
\usepackage{mathrsfs}

\usepackage{tikz}
\usetikzlibrary{arrows,decorations.markings}

\usepackage[pdfencoding=auto, colorlinks=true]{hyperref}
\hypersetup{
	pdfkeywords={}, linkcolor=blue, citecolor=blue, filecolor=blue, urlcolor=blue,
}

\newcommand\numberthis[1][]{%
	\refstepcounter{equation}%
	\ifx#1\empty\else\label{#1}\fi%
	\tag{\theequation}%
}

\numberwithin{equation}{section}
\allowdisplaybreaks

\renewcommand{\L}{\mathrm{L}}
\newcommand{\W}{\mathrm{W}}
\renewcommand{\H}{\mathrm{H}}
\renewcommand{\C}{\mathrm{C}}

\newcommand{\N}{\mathbb{N}}
\newcommand{\R}{\mathbb{R}}

\newcommand{\fflow}{\mathbf{f}}

\newcommand{\mflow}{\mathbf{m}}
\newcommand{\nflow}{\mathbf{n}}
\newcommand{\uflow}{\mathbf{u}}
\newcommand{\vflow}{\mathbf{v}}
\newcommand{\wflow}{\mathbf{w}}
\newcommand{\xflow}{\mathbf{x}}
\newcommand{\yflow}{\mathbf{y}}
\newcommand{\Cflow}{\mathbf{C}}

\newcommand{\Hflow}{\mathbf{H}}

\newcommand{\Lflow}{\mathbf{L}}

\newcommand{\Xflow}{\mathbf{X}}
\newcommand{\Vflow}{\mathbf{V}}
\newcommand{\Wflow}{\mathbf{W}}
\newcommand{\AMat}{\mathbb{A}}
\newcommand{\BMat}{\mathbb{B}}

\newcommand{\DMat}{\mathbb{D}}
\newcommand{\LMat}{\mathbb{L}}
\newcommand{\HMat}{\mathbb{H}}

\newcommand{\SMat}{\mathbb{S}}
\newcommand{\TMat}{\mathbb{T}}
\newcommand{\WMat}{\mathbb{W}}

\newcommand{\Sph}{{\SMat^{d-1}}}

\newcommand{\Rdd}{\R^{\text{d}\times \text{d}}}

\newcommand{\Hdiv}{\Hflow_{0,\text{div}}^1(\Omega)}

\newcommand{\Ldivp}[1]{\Lflow_\text{div}^{#1}(\Omega)}
\newcommand{\Zp}[1]{\text{Z}_{#1}}
\newcommand{\HkOS}[1]{\H^{#1}(\OmSph)}
\newcommand{\LpOS}[1]{\L^{#1}(\OmSph)}
\newcommand{\Vk}[1]{\Vflow^{#1}}

\newcommand{\surfaceLtr}{g} 
\newcommand{\Div}{\operatorname{div}}

\newcommand{\nablax}{\nabla_\xflow}
\newcommand{\Divx}{\Div_\xflow}
\newcommand{\Deltag}{\Delta_\surfaceLtr}
\newcommand{\nablag}{\nabla_\surfaceLtr}
\newcommand{\Divg}{\Div_\surfaceLtr}
\newcommand{\nablam}{\nabla_\mflow}

\newcommand{\sRot}{\mathcal{R}}	
\newcommand{\ddt}{\frac{\text{d}}{\dt}}

\newcommand{\lrcb}[1]{\left(#1\right)} 

\newcommand{\pos}[1]{\left[#1\right]_+}
\newcommand{\negat}[1]{\left[#1\right]_-}
\newcommand{\tr}[1]{\operatorname{tr}\left(#1\right)}

\newcommand{\IdMat}{\mathbb{I}}
\newcommand{\Je}[1]{\mathcal{J}_{\mollEps}\left( #1 \right)}
\newcommand{\JeC}{\mathcal{J}_{\mollEps}}

\newcommand{\JeMat}[1]{\mathbb{J}_{\mollEps}\left( #1 \right)}
\newcommand{\JeMatC}{\mathbb{J}_{\mollEps}}
\newcommand{\entrop}{\mathcal{F}}
\newcommand{\dentrop}{\entrop^\prime}
\newcommand{\ddentrop}{\entrop^{\prime\prime}}

\DeclareMathOperator*{\esssup}{ess\,sup}

\newcommand{\norm}[1]{\left\lVert#1\right\rVert}
\newcommand{\abs}[1]{\ensuremath{\left\vert#1\right\vert}}

\newcommand{\dx}{\, \text{d}\xflow}
\newcommand{\dt}{\,\text{dt}}
\newcommand{\dH}{\, \text{d}\mathscr{H}}	
\newcommand{\dHm}{\dH(\mflow)}
\newcommand{\dHmt}{\dH(\mflowtilde)}
\newcommand{\OmSph}{{\Omega\times\Sph}}
\newcommand{\nx}{\nflow_\xflow} 
\newcommand{\intT}{\int_0^T}

\newcommand{\intOmSph}{\int_{\OmSph}}
\newcommand{\intSph}{\int_{\Sph}}
\newcommand{\intOm}{\int_{\Omega}}

\newcommand{\RE}{\text{Re}}		
\newcommand{\DE}{\text{De}}		
\newcommand{\activity}{\alpha}		
\newcommand{\muPsi}{\mu_\psi}	
\newcommand{\muPsiov}{\overline{\mu}_\psi}
\newcommand{\pot}{\text{V}}		
\newcommand{\potov}{\overline{\pot}}
\newcommand{\potStr}{\xi}		
\newcommand{\viscFrac}{\gamma}	
\newcommand{\mollEps}{\varepsilon}	
\newcommand{\mflowtilde}{\widetilde{\mflow}}
\newcommand{\psiov}{\overline{\psi}}
\newcommand{\psihat}{\hat{\psi}}
\newcommand{\uflowhat}{\hat{\uflow}}

\newcommand{\Stress}{\DMat}		
\newcommand{\ndens}{\nu}    

\newcommand{\mom}{\mflow\otimes\mflow}
\newcommand{\weakstar}{\overset{*}{\rightharpoonup}}
\newcommand{\weak}{\rightharpoonup}
\newcommand{\taup}{{\tau,+}}
\newcommand{\taum}{{\tau,-}}

\newcommand{\CmollEps}{\overline{\C}_\mollEps}

\newcommand{\regL}{\text{L}}
\newcommand{\regDelta}{\delta}
\newcommand{\entropL}{\entrop^\regL}
\newcommand{\dentropL}{\lrcb{\entropL}^\prime}
\newcommand{\ddentropL}{\lrcb{\entropL}^{\prime\prime}}
\newcommand{\QL}{\text{Q}^\regL}
\newcommand{\QZL}{\text{Q}_0^\regL}

\newcommand{\muPsidLl}{\mu_{\psi, \regDelta, \regL, \lambda}}

\newcommand{\psiL}[1]{\psi_{\regL}^{#1}}

\newcommand{\muPsidL}[1]{\mu_{\psi, \regDelta, \regL}^{#1}}
\newcommand{\uL}[1]{\uflow_{\regL}^{#1}}
\newcommand{\potL}[1]{\pot_{\regL}^{#1}}

\newcommand{\PLtau}{\lrcb{\mathcal{P}_{\regL}^\tau}}

\newcommand{\omegaL}[1]{\omega_{\regL}^{#1}}
\newcommand{\StressL}[1]{\Stress_{\regL}^{#1}}
\newcommand{\dtau}{\partial_{t}^{\tau}}

\newtheorem{thm}{Proposition}[section]
\newtheorem{defi}[thm]{Definition}
\newtheorem{lmm}[thm]{Lemma}
\newtheorem{thrm}[thm]{Theorem}

\newtheorem{crllr}[thm]{Corollary}
\newtheorem{rmrk}[thm]{Remark}

\begin{document}
	
\title[Global Weak Solutions to an Inhomogeneous Doi Model]{Existence of Global Weak Solutions to an Inhomogeneous Doi Model for Active Liquid Crystals}
\author{Oliver Sieber}
\address{Department of Mathematics, Friedrich-Alexander-Universit\"at Erlangen-N\"urnberg, Cauerstr. 11, 91058 Erlangen, Germany}
\email{oliver.sieber@fau.de}
\date{\today{}}

\begin{abstract}
	In this paper, we consider an inhomogeneous Doi model which was introduced by W. E and P. Zhang [\textit{Meth. Appl. of Anal.}, 13 (2006), pp. 181-198]. We extend their model, which couples a Smoluchowski equation to a Navier-Stokes type equation, for active particles by introducing an additional stress tensor.
	Exploiting the energetic and entropic structure of the system, we establish the existence of global-in-time weak solutions in two and three space dimensions for both passive and active particles.
	In particular, our result holds for minimal regularity assumptions on the initial data and without restrictions on the Reynolds and Deborah number.
	
\end{abstract}

\subjclass[2010]{35Q35; 76A15; 76D03; 35Q92; 76D05; 82C31}

\keywords{Doi model; Active material; Navier-Stokes equation; Existence of weak solutions; Global entropy solution; Liquid crystals}

\maketitle

\section{Introduction}
\label{sec:intro}
For more than 130 years (cf. Lehmann \cite{lehmann1889krystalle}) liquid crystals inspire many scientists in a lot of ways.
Their field of applications ranges from  photonics and optics (e.g. liquid crystal displays (LCD)) over materials science (cf. \cite{palffy2007lc}) up to biophysics, for example to describe the cytoskeleton (cf. \cite{ahmadi2006activeLC, kierfeld2008active}). In the latter, there is often an activity mechanism involved which drives the system out of equilibrium and which will be discussed in more detail later.
Liquid crystals can be considered as a system which has mesomorphic states between an ordinary liquid and a crystal. In particular, such systems consist of rodlike polymers which interact pairwise and are characterized by long-range orientational order and not or only partially by positional order.
In the 1980s, Doi developed a model for rodlike molecules which describes the properties of liquid crystal polymers in a solvent (cf. \cite[Chapter 8]{doi1986theory}, \cite{doi1981molec_dynamics, degennes1993LC}). This model consists of a Smoluchowksi equation describing the evolution of a pseudo probability density of particles $\psi(\xflow, \mflow, t)$
with center of mass in $\xflow \in \Omega$ and orientation $\mflow \in \Sph$ at time $t \geq 0$, where $\Sph$ denotes the unit sphere and $\Omega \subset \R^d$ is a bounded domain for $d \in \{2, 3\}$. 
If the interaction of the rodlike polymers is very strong or if the concentration of them is very high, the rods line up with each other and build a nematic phase. This phase transition problem was first described by Onsager (cf. \cite{onsager1949colloidal, zhang2007review}).
In the homogeneous case, that is in the absence of effects of fluid dynamics, this model is often referred to as Doi-Onsager model whereas the model is referred to as Doi-Hess model in the inhomogeneous case.
In the past, these models have been the subject of many mathematical studies (cf. \cite{constantin2005nonlFP, constantin2007smolNast, constantin2008smolNaSt2d, liu2005classification, luo2004structure, otto2008velocityGrad, zhou2005equilibria}).\\
However, the models for inhomogeneous flow did not take distortional elasticity into account (cf. \cite{e_zhang_2006_kinetic, zhang2008existenceNewModel}). There have been many attempts to extend the theory for inhomogeneous flow to include distortional elasticity on a microscopic level (cf. \cite{marruci1991nematics, tsuji1998nematicLC, rey1998LCrheology, wang2002nematicLC}). Unfortunately, these models are all phenomenological in nature and hence, they consist of many unknown parameters which are very difficult to determine in experiments.
Moreover, they often lack consistency with other theories and among themselves (cf. \cite{e_zhang_2006_kinetic}).\\
Then, E and Zhang (cf. \cite{e_zhang_2006_kinetic}) extended the Doi model for homogeneous flow of rodlike liquid crystalline polymers to inhomogeneous flow and introduced a nonlocal intermolecular potential which results in an extra term in the form of an elastic body force to take into account distortional elasticity.\\
In this paper, we consider unflexible rodlike polymers with constant length $l > 0$, diameter $b > 0$, ratio $r:= \frac{b}{l} \ll 1$ (see Figure \ref{fig:rod}), and a mean number density $\ndens > 0$ given by
\begin{align*}
\ndens := \frac{1}{\abs{\Omega}} \intOmSph \psi \, \text{d}\mathscr{H}^{d-1} \dx.
\end{align*}
In particular, we restrict ourselves to solutions with a concentrated regime described by $\nu b l^2 > 1$ such as liquid crystals (cf. \cite[Chapter 10.1]{doi1986theory}).
\begin{figure}[!htbp]
    \begin{tikzpicture}[scale=0.75]
        \def\mx{0}	
        \def\my{0}	
        \def\b{0.5}
        \def\l{6*\b}
        \def\sqrtzwei{1.4142}
        \def\sqrtzweiinv{0.70711}
        
        \def\llx{\mx - \sqrtzweiinv*\b/2 - \sqrtzweiinv*\l} 	
        \def\lly{\my + \sqrtzweiinv*\b/2 - \sqrtzweiinv*\l}		
        
        \def\ulx{\mx - \sqrtzweiinv*\b/2 + \sqrtzweiinv*\l}		
        \def\uly{\mx + \sqrtzweiinv*\b/2 + \sqrtzweiinv*\l}		
        
        \def\lrx{\mx + \sqrtzweiinv*\b/2 - \sqrtzweiinv*\l}		
        \def\lry{\my - \sqrtzweiinv*\b/2 - \sqrtzweiinv*\l}		
        
        \def\urx{\mx + \sqrtzweiinv*\b/2 + \sqrtzweiinv*\l}		
        \def\ury{\my - \sqrtzweiinv*\b/2 + \sqrtzweiinv*\l}		
        
        \def\radellipsex{\b/2}	
        \def\radellipsey{\b/4}	
        
        \def\disttorod{\b*2/3}	

        \draw (\llx, \lly) -- (\ulx, \uly);
        \draw (\lrx, \lry) -- (\urx, \ury);
        
        \node[circle,fill=black,inner sep=0pt,minimum size=5pt,label=below left:{$\mathbf{x}$}] (a) at (\mx, \my) {};
        
        \draw[rotate around={-45:(\llx/2 + \lrx/2, \lly/2 + \lry/2)}] (\llx/2 + \lrx/2, \lly/2 + \lry/2) ellipse ({\radellipsex} and {\radellipsey});
        \draw[rotate around={-45:(\urx, \ury)}] (\urx, \ury) arc(0:180: {\radellipsex} and {\radellipsey});
        
        \draw[arrows={angle 90 - angle 90}] (\llx - \disttorod, \lly + \disttorod) -- (\ulx - \disttorod, \uly + \disttorod);
        \draw[arrows={angle 90 - angle 90}] (\ulx + \disttorod, \uly + \disttorod) -- (\urx + \disttorod, \ury + \disttorod);
        
        \node (nameL) at (\mx - \sqrtzweiinv*\b - \disttorod - \b/2, \my + \sqrtzweiinv*\b + \disttorod + \b/2) {$l$};
        \node (nameB) at (\ulx/2 + \urx/2 + \disttorod + \b/2, \uly/2 + \ury/2 + \disttorod + \b/2) {$b$};
    \end{tikzpicture}
    \caption{Rodlike polymer with length $l$, diameter $b$, and center of mass $\xflow$.}
    \label{fig:rod}
\end{figure}
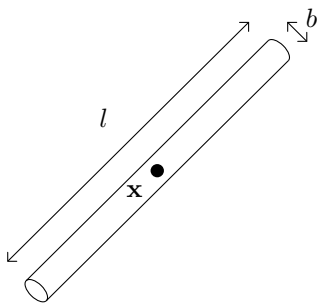
The model of E and Zhang in dimensional form (cf. \cite{e_zhang_2006_kinetic}) coupling a Smoluchowski equation to a Navier-Stokes type equation reads as follows for a given time $T > 0$:
\begin{subequations}
	\label{intro:eq:dimensional}
	\begin{align}
	\begin{split}
	\label{intro:eq:psi:dimensional:srot}
	&\partial_t \psi + \Divx(\uflow \psi) - \frac{1}{ k_B \mathcal{T}} \Divx\lrcb{ \psi\lrcb{ D_\parallel \mom + D_\perp (\IdMat - \mom) }\nablax \muPsiov } \\
	&\qquad - \frac{D_r}{ k_B \mathcal{T}} \sRot \cdot \lrcb{ \psi \sRot \muPsiov }  = -  \sRot \cdot \lrcb{ \psi \mflow \times (\nablax \uflow \mflow)  } \quad \text{ in } \OmSph\times(0,T],
	\end{split}\\
	\begin{split}
	\label{intro:eq:u:dimensional}
	&\rho \lrcb{ \partial_t \uflow + (\uflow \cdot \nablax) \uflow } + \nablax p = - \intSph \psi \nablax \muPsiov \, \text{d}\mathscr{H}^{d-1} + \Divx(\overline{\TMat}) \quad \text{ in } \Omega \times (0,T].
	\end{split}
	\end{align}
\end{subequations}
At this point let us refer to Section \ref{sec:model} for details about this model.
As the degrees of freedom of the microscopic equation are quite high, there have been many attempts to derive macroscopic equations from the Doi model (a good overview over the different models from a mathematical point of view can be found in \cite{emmrich2018LC, ball2017LC}).
In particular, the interest in the Ericksen-Leslie equations (cf. \cite{ericksen1961convervationLC, ericksen1991orientationLC, leslie1968LC}), the Q-Tensor equations (cf. \cite{degennes1993LC, majumdar2010LC, beris1994thermodynamics}), and the mathematical analysis of these models (cf. \cite{lin1995dissipativeLC, lin1996dynamicLC, lin2000ericksenLeslie, cavaterra2013LC, ball2010LC, paicu2011QTensor, paicu2012QTensor, huang2015LC, abels2014QTensor}) has been growing significantly over the last years.
Kuzuu and Doi's (cf. \cite{kuzuu1983LC}) formal derivation of the Ericksen-Leslie equation from the Doi-Onsager equation in the homogeneous case has been extended to the inhomogeneous case from systematic asymptotic analysis of \eqref{intro:eq:dimensional} (cf. E and Zhang \cite{e_zhang_2006_kinetic}). Moreover, Wang et al. (cf. \cite{wang2015deborahLimit}) established a rigorous derivation of the Ericksen-Leslie equation from \eqref{intro:eq:dimensional}. Furthermore, Han et al. (cf. \cite{han2015microToMacro}) derived a new Q-Tensor model from \eqref{intro:eq:dimensional}, which obeys the energy dissipation law, by applying the Bingham-closure (cf. \cite{bingham1974}). Therefore, as the model \eqref{intro:eq:dimensional} is not just a fundamental model itself but a starting point for the derivation of many equations, which find a lot of interest these days, it is important to study it analytically.

Chen and Liu (\cite{chen_liu_2013_entropy_solution}) proved the existence of global-in-time weak solutions to \eqref{intro:eq:dimensional} including active stresses induced by active particles (see Section \ref{sec:model} for more details regarding activity) in two and three space dimensions in the case that the friction constant is zero, that is $\xi_r = 0$ (see \eqref{eq:viscous_stress_p}) and that the interaction potential equals zero, that is $\pot[\psi] \equiv 0$ (see \eqref{def:pot}).
Zhang and Zhang (cf. \cite{zhang2008existenceNewModel}) considered the model \eqref{intro:eq:dimensional}, periodic domains, and high regularity assumptions on the initial data.
They proved the existence of local solutions in two and three space dimensions (cf. \cite[Theorem 1.1]{zhang2008existenceNewModel}). 
In the case of two space dimensions, they additionally showed the existence of global-in-time weak solutions (in the sense that $0 < T < \infty$ is given) (cf. \cite[Remark 1.2]{zhang2008existenceNewModel}).
Moreover, they established the existence of global weak solutions (in the sense that $[0,T) = [0, \infty)$) if the Deborah number and the Reynolds number are sufficiently small for $d \in \{2, 3\}$ (cf. \cite[Theorem 1.2]{zhang2008existenceNewModel}).
However, proving the existence of global-in-time weak solutions to \eqref{intro:eq:dimensional} (for passive or active particles in the sense that $0 < T < \infty$ is given) for three space dimensions without assumptions on the Reynolds and Deborah number was stated as an open problem by Emmrich et al. (cf. \cite[p.42]{emmrich2018LC}).

The main contribution of this paper is the proof of the latter problem. In Main Theorem \ref{thrm:main}, we establish the existence of global-in-time weak solutions to \eqref{intro:eq:dimensional} for both passive and active particles (in the sense that $0 < T < \infty$ is given) in two and three space dimensions under minimal regularity assumptions on the initial data. Moreover, our result holds for arbitrary values of the Reynolds and Deborah number. To be precise, we consider the model $(\mathcal{P})$, which is introduced in Section \ref{sec:model}.

The outline of this paper is as follows: 
First, we complete Section \ref{sec:intro} by introducing the model $(\mathcal{P})$, which is an extension of \eqref{intro:eq:dimensional} for active particles.
Then, we present some notation, function spaces, and frequently used formula in Section \ref{sec:preliminares}. To get started, we derive an entropy estimate \eqref{formal:entropy_est} for the model $(\mathcal{P})$ in Section \ref{sec:formal_entropy_estimate} on a formal level.
As we want to make our computations rigorous, we present a discrete-in-time scheme $\PLtau$ (see \eqref{def:PLtau}) in Section \ref{sec:discrete_in_time_approx}, where $\tau > 0$ is a small time increment and $\regL > 0$ is a cut-off parameter from above. The regularization parameter $\regL$ is used to establish the existence of discrete-in-time solutions. Nevertheless, future bounds in energy estimates shall be independent of $\regL$.\\
To show the existence of a solution to $\PLtau$ in Section \ref{sec:ex_of_sols_to_PLtau}, we introduce a fixed-point scheme in Definition \ref{def:fixed_point_iteration}, prove in Lemma \ref{lmm:fixed_point_iter} that the scheme is well defined, and finally apply Schaefer's fixed-point theorem in Lemma \ref{lmm:existence_fp} to complete the proof.\\
Having established the existence of a solution to $\PLtau$, we devote Section \ref{sec:uniform_entropy_estimate} to the derivation of an entropy estimate in Lemma \ref{lmm:time_discrete_entropy} which mimics the formal estimate \eqref{formal:entropy_est}. Unfortunately, the regularity results of the entropy estimate \eqref{est:discrete_entropy} are too weak to pass to the limit in $\PLtau$ as $\tau \searrow 0$ and $\regL \nearrow \infty$.
Moreover, testing the Smoluchowski equation with $\psi$ is fraud with difficulties to gain more regularity. However, we shall circumvent this problem by deriving an equality for the polymer number density
\begin{align*}
\omega(\xflow, t) := \intSph \psi(\xflow, \mflow, t) \, \text{d}\mathscr{H}^{d-1}(\mflow) \quad \text{ for } (\xflow, t) \in \Omega \times [0,T]
\end{align*}
and show regularity results for $\omega$ which turn out to be very useful to improve the regularity of $\psi$ afterwards. Such a method is often applied to microscopic macroscopic models (cf. \cite{barrett_sueli_2011_polymers_I, chen_liu_2013_entropy_solution, gruen_metzger_2015microMacro}). The derivation of regularity results for the macroscopic discrete-in-time polymer number density is realized in Section \ref{sec:particle_density}.\\
Then, we establish time regularity results for $\psi$ and $\uflow$ in Section \ref{sec:time_regularity}. To pass to the limit in Section \ref{sec:passage_limit} as $\tau \searrow 0$ and $\regL \nearrow \infty$, we exploit our regularity results to prove strong convergence for suitable subsequences. Instead of using common results of Aubin-Lions or Simon (cf. \cite{lions1969quelques, simon1986compact}), we apply a compactness result of Dreher and J\"ungel (cf. \cite{dreher2012compact}), which relies on hypothesis on time translation and avoids the construction of linear interpolation functions in time.

\subsection{Doi model for active liquid crystals}
\label{sec:model}
Before we extend the model for active particles, let us describe the system \eqref{intro:eq:dimensional}  in more detail.
The velocity field is denoted by $\uflow: \Omega \times [0,T] \to \R^d$, $p: \Omega \times (0,T] \to \R$ is the pressure, $k_B$ is the Boltzmann constant, $\mathcal{T}$ is the absolute temperature, $\ndens > 0$ is the mean number density of polymers, $D_r = \frac{k_B \mathcal{T}}{\xi_r}$ is the rotary diffusivity, $\xi_r > 0$ is the friction coefficient, $\rho > 0$ is the fluid density, $\mathscr{H}^{d-1}$ is the Hausdorff measure, and $D_\parallel > 0$ and $D_\perp > 0$ are the translational diffusion coefficients parallel and normal to the orientation of the rods.
Moreover, it is always assumed that $\mflow \in \Sph$.
Furthermore, $\nablax$ and $\Divx$ denote the gradient and divergence operator with respect to $\xflow \in \Omega$. The rotational gradient operator $\sRot$ on the unit sphere is defined by $\sRot = \mflow \times \nablam$, where $\nablam$ is the gradient with respect to $\mflow \in \Sph$. Moreover, the chemical potential $\muPsiov$ is given by
\begin{align}
\label{def:muPsi}
\muPsiov := k_B \mathcal{T} \dentrop(\psi) + k_B \mathcal{T} \pot[\psi],
\end{align}
where $\entrop: \R_{\geq 0} \ni s \mapsto s (\log(s) - 1) + 1 \in \R_{\geq 0}$ is an entropy functional.

The interaction potential $\potov[\psi] : \OmSph\times[0,T] \to \R$ modeling the effects of alignment is defined by
\begin{align}
\label{def:pot}
\pot[\psi](\xflow, \mflow, t) :=  \potStr \Je{ \intSph \psi(\cdot, \mflowtilde, t) K(\mflow, \mflowtilde) \, \text{d}\mathscr{H}^{d-1}(\mflowtilde)} (\xflow)
\end{align}
for $(\xflow, \mflow, t) \in \OmSph \times [0,T]$, where the dimensionless parameter $\potStr > 0$ denotes the strength of the interaction potential, $\JeC$ is an isotropic mollifier which is defined below, and $K(\mflow, \mflowtilde)$ is an interaction kernel for $\mflow, \mflowtilde \in \Sph$ which should have small values for $\mflow$ and $\mflowtilde$ being aligned and large values otherwise. In the original Onsager model (cf. \cite{onsager1949colloidal}), the interaction kernel is given by
\begin{align*}
K(\mflow, \mflowtilde) = \abs{\mflow \times \mflowtilde} \quad \text{ for } \mflow, \mflowtilde \in \Sph.
\end{align*}
Another common choice is the Maier-Saupe kernel (cf. \cite{maier_saupe_1958})
\begin{align}
\label{def:interaction_kernel}
K(\mflow, \mflowtilde) := \abs{\mflow \times \mflowtilde}^2 = 1 - (\mflow \cdot \mflowtilde)^2 \quad \text{ for } \mflow, \mflowtilde \in \Sph,
\end{align}
which has also been used by E and Zhang (cf. \cite{e_zhang_2006_kinetic}).
In this paper, we restrict ourselves to the Maier-Saupe interaction kernel \eqref{def:interaction_kernel}.
Moreover, the isotropic mollifier $\JeC: \L^1(\Omega) \to \W^{1, \infty}(\Omega)$ in \eqref{def:pot} is defined by
\begin{align}
\label{def:mollifier}
\Je{f}(\xflow) := \intOm \zeta_\mollEps(\xflow - \yflow) f(\yflow) \, \text{d}\yflow \quad \text{ for } f \in \L^1(\Omega),
\end{align}
where $\zeta_\mollEps (\xflow) := \mollEps^{-d} \zeta(\mollEps^{-1} \xflow)$ for $\xflow \in \Omega$ and $\mollEps > 0$.
Hereby, the nonnegative, rotationally symmetric, and compactly supported function $\zeta \in \C_0^\infty(\R^d)$ is given by $\zeta(\xflow) := c \exp(-\frac{1}{1 - \abs{\xflow}^2})$ if $\abs{\xflow} < 1$ and $\zeta(\xflow) = 0$ otherwise. In particular, the constant $c > 0$ is chosen such that $\int_{\R^d} \zeta(\xflow) \dx = 1$.
Therefore, on noting a property of the mollifier (see \eqref{mollifier:fg}), we have that $\muPsiov$ is the first variation of the energy functional
\begin{align*}
\mathcal{E}(\psi)(t) &:= k_B \mathcal{T} \intOmSph\entrop(\psi)(\xflow, \mflow, t) \,  \text{d}\mathscr{H}^{d-1}(\mflow) \dx \\
&\quad + \frac{ k_B \mathcal{T}}{2} \intOmSph  \psi(\xflow, \mflow, t) \pot[\psi](\xflow, \mflow, t) \,  \text{d}\mathscr{H}^{d-1}(\mflow) \dx
\end{align*}
for $t \in [0,T]$.

Now, let us describe the physical interpretation of the terms in \eqref{intro:eq:dimensional} in more detail.
The term on the right-hand side of \eqref{intro:eq:psi:dimensional:srot} describes the rotation of particles induced by the macroscopic velocity gradient $\nablax \uflow$ (note \eqref{intro:eq:psi:dimensional:nablag} for the rewritten form of this term).
As illustrated in Figure \ref{fig:rotated_rod}, a rod $\mflow \in \Sph$ is rotated by the gradient of a macroscopic flow $\nablax \uflow$. However, as the rod cannot elongate and has to maintain its constant length, the rotated rod $\nablax \uflow \mflow$ is projected onto the tangential space of the unit sphere.
\begin{figure}[!htbp]
    \centering
    \begin{tikzpicture}[scale=1.0]
    
    \def\ulx{1.374}     
    \def\uly{1.624}     
    
    \def\urx{\uly}      
    \def\ury{\ulx}      
    
    \def\llx{-\urx}     
    \def\lly{-\ury}     
    
    \def\lrx{-\ulx}     
    \def\lry{-\uly}     
    
    \def\sulx{2.002}    
    \def\suly{0.720}    
    
    \def\surx{2.094}    
    \def\sury{0.378}    
    
    \def\sllx{-\surx}   
    \def\slly{-\sury}   
    
    \def\slrx{-\sulx}   
    \def\slry{-\suly}   
    
    \def\umx{\ulx/2 + \urx/2}       
    \def\umy{\uly/2 + \ury/2}       
    
    \def\lmx{\llx/2 + \lrx/2}       
    \def\lmy{\lly/2 + \lry/2}       
    
    \def\sumx{\sulx/2 + \surx/2}    
    \def\sumy{\suly/2 + \sury/2}    
    
    \def\slmx{\sllx/2 + \slrx/2}    
    \def\slmy{\slly/2 + \slry/2}    
    
    \def\mx{\ulx/4 + \urx/4 + \llx/4 + \lrx/4}    
    \def\my{\uly/4 + \ury/4 + \lly/4 + \lry/4}    
    
    \draw (\sllx, \slly) -- (\sulx, \suly) -- (\surx, \sury) -- (\slrx, \slry) -- cycle;
    
    \draw[fill=white] (\llx, \lly) -- (\ulx, \uly) -- (\urx, \ury) -- (\lrx, \lry) -- cycle;
    
    \draw[arrows={- angle 90}] (\lmx+0.1, \lmy+0.1) -- (\umx-0.1, \umy-0.1) node[pos=0.85, left=16pt, scale=0.75, rotate=45] {$\mathbf{m}$};
    
    \draw[arrows={- angle 90}, line width=1pt] (\umx, \umy) -- (\sumx, \sumy) node[pos=0.85, below=12pt, scale=0.75, rotate=15] {$\mathbf{(\mathbb{I} - \mathbf{m}\otimes\mathbf{m})} \nabla_{\mathbf{x}}\mathbf{u}\, \mathbf{m}$};
    
    \draw[dashed, rotate around={15:(\sumx, \sumy)}] (\sumx, \sumy) -- (\sumx + 4, \sumy);
    
    \draw[arrows={- angle 90}, line width=1pt, rotate around={-10:(\umx, \umy)}] (\umx, \umy) -- (\umx+2.5, \umy) node[pos=1.0, below=7pt, scale=0.75, rotate=15] {$\quad \nabla_{\mathbf{x}} \mathbf{u}\, \mathbf{m}$};
    
    \node[circle,fill=black,inner sep=0pt,minimum size=3pt] at (\mx, \my) {};
    \end{tikzpicture}
    \caption{A rod $\mflow$ is rotated by a macroscopic velocity gradient $\nablax \uflow$.}
    \label{fig:rotated_rod}
\end{figure}
The stress tensor on the right-hand side of \eqref{intro:eq:u:dimensional}
\begin{align*}
\overline{\TMat} := \overline{\TMat}_{\text{visc, s}} + \overline{\TMat}_{\text{visc, p}} + \overline{\TMat}_{\text{elast}}
\end{align*}
is composed of viscous stresses $\overline{\TMat}_{\text{visc}} := \overline{\TMat}_{\text{visc, s}} + \overline{\TMat}_{\text{visc, p}}$ and elastic stresses $\overline{\TMat}_{\text{elast}}$. Hereby,
\begin{align*}
\overline{\TMat}_{\text{visc, s}} := \eta_s \lrcb{\nablax \uflow + (\nablax \uflow)^T}
\end{align*}
models the viscous stresses coming from the solvent, where $\eta_s > 0$ denotes the solvent viscosity.
The viscous stresses $\overline{\TMat}_{\text{visc, p}}$ are due to the friction caused by the motion of rodlike polymers. They are related to the hydrodynamic energy dissipation $W$ (cf. \cite[(8.117)]{doi1986theory}) by
\begin{align}
\label{intro:eq:dissip}
W = \int_0^T \intOm \nablax\uflow : \overline{\TMat}_{\text{visc, p}} \dx\dt.
\end{align}
On noting Figure \ref{fig:rotated_rod}, the velocity of a rod relative to the fluid can be computed by
\begin{align*}
\uflow_{\text{rod, fluid}} = (\IdMat - \mom) \nablax \uflow \mflow - \nablax \uflow \mflow = - (\mom) \nablax \uflow \mflow.
\end{align*}
Moreover, the frictional force exerted on a rod is given by
\begin{align*}
\fflow_{\text{rod, frict}} = \frac{ \xi_r}{2} \uflow_{\text{rod, fluid}},
\end{align*}
where $\xi_r > 0$ is the friction constant. Therefore, the work done by the frictional force is calculated as
\begin{align*}
W &= \intT \intOm \intSph \psi \fflow_{\text{rod, frict}} \cdot \uflow_{\text{rod, fluid}} \, \text{d}\mathscr{H}^{d-1} \dx \dt \\
&= \frac{ \xi_r}{2} \intT\intOm\intSph \psi \lrcb{ \nablax \uflow : \mom }^2 \, \text{d}\mathscr{H}^{d-1}(\mflow)\dx\dt.
\end{align*}
Hence, on noting \eqref{intro:eq:dissip}, we have that
\begin{align}
\label{eq:viscous_stress_p}
\overline{\TMat}_{\text{visc, p}} = \frac{ \xi_r}{2} \intSph \psi (\nablax \uflow : \mom) \mom \, \text{d}\mathscr{H}^{d-1}(\mflow).
\end{align}
The body force term (the first term on the right-hand side of \eqref{intro:eq:u:dimensional}) and the elastic stress tensor $\overline{\TMat}_{\text{elast}}$ are related to the change in energy of $\mathcal{E}(\psi)$. In particular, they are inevitable for thermodynamic consistency and they have been derived by a virtual work principle (cf. \cite[Chapter 8.6]{doi1986theory}, \cite[p. 185f]{e_zhang_2006_kinetic}) resulting in an elastic stress tensor given by
\begin{align*}
\overline{\TMat}_{\text{elast}} := - \intSph \psi \lrcb{\mflow \times \sRot \muPsiov} \otimes \mflow \, \text{d}\mathscr{H}^{d-1}(\mflow).
\end{align*}


In our opinion, it is more convenient to work with the surface gradient operator $\nablag$ and the surface divergence operator $\Divg$ on $\Sph$ instead of $\sRot$. 
Noting $\sRot = \mflow \times \nablam$ and $\nablag = (\IdMat - \mom) \nablam$, we see that \eqref{intro:eq:psi:dimensional:srot} can equivalently be written as
\begin{align}
\label{intro:eq:psi:dimensional:nablag}
\begin{split}
&\partial_t \psi + \Divx(\uflow \psi) - \frac{1}{ k_B \mathcal{T}} \Divx\lrcb{ \psi\lrcb{ D_\parallel \mom + D_\perp (\IdMat - \mom) }\nablax \muPsiov } \\
&\quad - \frac{D_r}{ k_B \mathcal{T}} \Divg\lrcb{ \psi \nablag \muPsiov }  = -  \Divg\lrcb{ \psi (\IdMat - \mom) \nablax \uflow \mflow } \quad \text{ in } \OmSph\times(0,T]
\end{split}
\end{align}
and that the elastic stress tensor can be rewritten as
\begin{equation}
\label{intro:eq:stress_sRot_divg}
\begin{aligned}
\overline{\TMat}_{\text{elast}} &= - \intSph \psi (\mflow \times \sRot \muPsiov) \otimes \mflow \, \text{d}\mathscr{H}^{d-1}(\mflow)  \\
&=  \intSph \psi \left[ (\IdMat - \mom) \nablag \muPsiov \right] \otimes \mflow \, \text{d}\mathscr{H}^{d-1}(\mflow) .
\end{aligned}
\end{equation}

Now, let us come back to active liquid crystals which are systems that are far from equilibrium (cf. \cite{gao2017activeNematics}). The range of active nematic systems includes cytoskeletal filaments (cf. \cite{kierfeld2008active}), bacteria (cf. \cite{darnton2010swarming}), dense suspensions of microswimmers (cf. \cite{wensink2012livingFluids}), and microtubule bundles (cf. \cite{sanchez2012spontaneous}).
Instead of modeling passive particles, we now consider active rodlike polymers which induce active stresses described by an additional stress tensor $\overline{\TMat}_{\text{act}}$. Following the lines in \cite{gao2017analytical}, we model this stress tensor: We assume that each rod induces a symmetric surface flow (see Figure \ref{fig:active_flow}, Figure \ref{fig:rod})
\begin{align*}
\uflow_{\text{act}}(s) := \text{sgn}(s) u_\text{act} \mflow, \quad - \frac{l}{2} \leq s \leq \frac{l}{2},
\end{align*}
where $u_\text{act}$ is the signed surface flow speed, $\mflow$ is the unit orientation vector of the rod, and $\text{sgn}(\cdot)$ is the sign function.
\begin{figure}[!htbp]
    \centering
    \begin{tikzpicture}[scale=1.0]
    \def\mx{0}	
    \def\my{0}	
    \def\b{3/4}
    \def\l{6*\b}
    \def\disttobase{20pt}
    \def\eps{2pt}	
    
    \draw [|-|] (\mx - \l/2, \my) -- (\mx, \my) node[pos=0, below, scale=0.75] {$\mathbf{x} - \frac{l}{2}$} node[pos=1, below=3pt, scale=0.75] {$\mathbf{x}$};
    \draw [|-|] (\mx, \my) -- (\mx + \l/2, \my) node[pos=1, below, scale=0.75] {$\mathbf{x} + \frac{l}{2}$};
    
    \draw [->, line width=1pt, red] (\mx - \eps, \my - \disttobase) -- (\mx - \l/2, \my - \disttobase) node[pos=0.5, below=3pt, scale=0.75] {$- u_\text{act}\, \mathbf{m}$};
    \draw [->, line width=1pt, red] (\mx + \eps, \my - \disttobase) -- (\mx + \l/2, \my - \disttobase) node[pos=0.5, below=3pt, scale=0.75] {$u_\text{act}\, \mathbf{m}$};
    
    \draw[arrows={- triangle 90}, line width=1pt, blue] (\mx - \l/2, \my + \disttobase*3/4) -- (\mx + \l/2, \my + \disttobase*3/4) node[pos=0.5, above=3pt, scale=0.75] {$\mathbf{m}$};
    \end{tikzpicture}
    \caption{Symmetric surface flow induces by an active rod.}
    \label{fig:active_flow}
\end{figure}
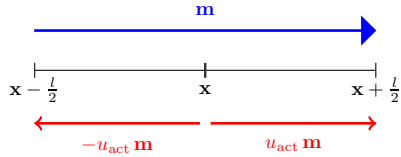
Hereby, a rod produces an extensional straining flow along itself for $u_\text{act} > 0$ and a compressive one otherwise. Moreover, let us describe the position of a rod by its centerline $\Xflow(s) := \xflow + s \mflow$, $-\frac{l}{2} \leq s \leq \frac{l}{2}$, where $\xflow$ denotes the center of mass of a rod (see Figure \ref{fig:rod}). Then, the active stress caused by a single rod can be computed as follows on assuming that $\xflow = 0$ (cf. \cite{gao2017analytical, saintillan2013nonlinear}):
\begin{align*}
\overline{\TMat}_{\text{act, rod}}
&= - \frac{4 \pi \eta}{\abs{\log(e r^2)}} \int_{-\frac{l}{2}}^{\frac{l}{2}} \uflow_{\text{act}}(s) \otimes \Xflow(s) \,\text{d}s \\
&= - \frac{\pi \eta u_\text{act} l^2}{\abs{\log(e r^2)}} \mom,
\end{align*}
where $\eta := \eta_s + \eta_p > 0$ is the total viscosity, $\eta_p := \nu \xi_r$ is the polymer viscosity, $e := \exp(1)$, and $r = \frac{b}{l}$ (see Figure \ref{fig:rod}).
Hence, the active stress tensor due to the activity of the rodlike polymers is given by
\begin{align}
\label{intro:def:Tact_dimensional}
\overline{\TMat}_{\text{act}} = - \frac{\pi \eta u_\text{act} l^2}{\abs{\log(e r^2)}} \intSph \psi \lrcb{ \mom - \frac{1}{d}\IdMat } \, \text{d}\mathscr{H}^{d-1}(\mflow),
\end{align}
where the subtraction \eqref{intro:def:Tact_dimensional} modifies solely the pressure and has no influence on the flow. In particular, the modification in \eqref{intro:def:Tact_dimensional} is equivalent to subtracting the trace of the tensor to make it traceless.\\
%
%
It will turn out, note for example the formal energy estimate \eqref{formal:entropy_est}, that the additional active stresses may be seen as an internal source of energy as the thermodynamic equilibrium does not hold any longer in the sense that the total energy at a time $t_2 > t_1$ equals the total energy at time $t_1$ when there is no work done by external forces. This can be explained by the ability of active liquid crystals to convert energy from the local environment to mechanical work (cf. \cite{gao2017activeNematics}).\\
Next, to derive the model in nondimensionalised form, we follow the lines in \cite{e_zhang_2006_kinetic}: Let $L_0 > 0$ be a typical size of the flow region, $V_0 > 0$ the typical velocity scale, and $T_0 := \frac{L_0}{V_0}$ a typical convective time scale. 
Moreover, we scale the pseudo probability density $\psi$ by $\frac{1}{\nu}$, where $\nu > 0$ is the mean number density.
We introduce the Deborah number, which is defined by the ratio of the orientational diffusion time scale of the rods and the convective time scale of the fluid, given by
\begin{align*}
\DE := \frac{ \frac{\xi_r}{k_B \mathcal{T}} }{\frac{L_0}{V_0}} = \frac{\xi_r V_0}{k_B \mathcal{T} L_0}.
\end{align*}
Moreover, we note that $\eta_p = \ndens \xi_r$, $\eta = \eta_s + \eta_p > 0$ and set $\viscFrac := \frac{\eta_s}{\eta} \in (0,1)$. Then, the Reynolds number is defined by
\begin{align*}
\RE := \frac{\rho V_0 L_0}{\eta}.
\end{align*}
Furthermore, we set $\mollEps := \frac{l}{L_0} > 0$, where $l>0$ denotes the length of rods (see Figure \ref{fig:rod}).
We define the dimensionless translational diffusion parameters by
\begin{align*}
\hat{D}_\parallel := \frac{\xi_r}{l^2 k_B \mathcal{T}} D_\parallel \ \text{ and } \ \hat{D}_\perp := \frac{\xi_r}{l^2 k_B \mathcal{T}} D_\perp.
\end{align*}
It usually holds that $\hat{D}_\parallel > \hat{D}_\perp > 0$ (\cite[p. 296f]{doi1986theory}) and in most of the cases one has that $\hat{D}_\parallel = 2 \hat{D}_\perp$ (\cite[p. 296f]{doi1986theory}). From an analytical point of view, it is important that the tensor
\begin{align*}
\hat{D}_\parallel \mom + \hat{D}_\perp (\IdMat - \mom)
= \hat{D}_\perp \IdMat + (\hat{D}_\parallel - \hat{D}_\perp) \mom
\end{align*}
is symmetric and positive definite, which is true if $\hat{D}_\parallel \geq \hat{D}_\perp > 0$. Hence, we set $\hat{D}_\parallel = \hat{D}_\perp = 1$ for ease of notation in this paper.
On noting \eqref{intro:def:Tact_dimensional}, we introduce the dimensionless activity parameter
\begin{align*}
\activity := - \frac{\pi \eta l^2 u_\text{act}}{\abs{\log(e r^2)} k_B \mathcal{T}}.
\end{align*}

Thus, the active model in nondimensionalised form (without relabelling $\psi$, $\uflow$, $p$, $\xflow$, and $t$ for ease of notation) reads as follows:\\
\hspace*{2ex}Given $T > 0$, find $\uflow: \Omega \times [0,T] \ni (\xflow, t) \mapsto \uflow(\xflow, t) \in \R^d$,\\
\hspace*{2ex}$p: \Omega \times (0,T] \ni (\xflow, t) \mapsto p(\xflow, t) \in \R$,\\
\hspace*{2ex}and $\psi: \OmSph\times[0,T] \ni (\xflow, \mflow, t) \mapsto \psi(\xflow, \mflow, t) \in \R$ such that
\begin{subequations}
    \label{eq:active:stokes}
    \begin{alignat}{3}
    \partial_t \uflow + (\uflow \cdot \nablax)\uflow + \nablax p &= - \frac{1 - \viscFrac}{\RE\DE} \intSph \psi \nablax \muPsi \, \text{d}\mathscr{H}^{d-1} \label{eq:active:u} \\
    &\quad + \Divx(\TMat)  \quad && \text{in } \Omega \times (0,T], \nonumber \\
    \Divx(\uflow) &= 0 \quad &&\text{in } \Omega \times (0,T], \label{eq:active:divu} \\
    \uflow &= 0 \quad &&\text{on } \partial\Omega \times (0,T], \label{eq:active:bc:u} \\
    \uflow(\xflow, 0) &= \uflow_0(\xflow) \quad &&\forall \xflow \in \Omega \label{eq:active:init:u}
    \end{alignat}
\end{subequations}
and
\begin{subequations}
    \label{eq:active:smoluchowski}
    \begin{align}
    & \partial_t \psi + \uflow \cdot \nablax \psi - \frac{\mollEps^2}{\DE} \Divx (\psi \nablax \muPsi) - \frac{1}{\DE} \Divg \lrcb{ \psi \nablag \muPsi } \nonumber \\
    &\quad = - \Divg \lrcb{ (\IdMat - \mom) \nablax \uflow \mflow \psi } \qquad \text{in } \Omega \times \Sph \times (0,T], \label{eq:active:psi} \\
    &\muPsi := \dentrop(\psi) + \pot[\psi] \qquad\qquad\qquad\qquad\ \ \, \text{in } \Omega \times \Sph \times (0,T], \label{eq:active:muPsi}\\
    &\psi \nablax \muPsi \cdot \nx = 0 \qquad\qquad\qquad\qquad\qquad\ \, \text{on } \partial \Omega \times \Sph \times (0,T], \label{eq:active:bc:psi} \\
    &\psi(\xflow, \mflow, 0) = \psi_0(\xflow, \mflow) \qquad\qquad\qquad\quad\ \ \ \forall (\xflow, \mflow) \in \OmSph, \label{eq:active:init:psi}
    \end{align}
\end{subequations}
where $\nx : \partial\Omega \to \R^d$ denotes the outward normal to $\partial\Omega$ and the stress tensor in the Navier-Stokes type equation is given by
\begin{align*}
\TMat := \TMat_{\text{visc, s}} + \TMat_{\text{visc, p}} + \TMat_{\text{elast}} + \TMat_{\text{act}}
\end{align*}
with
\begin{align*}
\TMat_{\text{visc, s}} &:= \frac{\viscFrac}{\RE} \lrcb{\nablax \uflow + (\nablax \uflow)^T}, \\
\TMat_{\text{visc, p}} &:= \frac{1 - \viscFrac}{2 \RE}  \intSph \psi (\nablax \uflow : \mom) \mom \, \text{d}\mathscr{H}^{d-1}(\mflow), \\
\TMat_{\text{elast}} &:= \frac{1 - \viscFrac}{\RE\DE}  \intSph \psi \left[ (\IdMat - \mom) \nablag \muPsi \right] \otimes \mflow \, \text{d}\mathscr{H}^{d-1}(\mflow), \\
\TMat_{\text{act}} &:= \frac{\activity(1 - \viscFrac)}{\RE\DE}  \intSph \psi \lrcb{\mom - \frac{1}{d}\IdMat} \, \text{d}\mathscr{H}^{d-1}(\mflow).
\end{align*}

%
The initial data $\uflow_0$ and $\psi_0$ will be specified later.

\begin{defi}
    The coupled problem \eqref{eq:active:stokes} - \eqref{eq:active:smoluchowski} will be referred to as model $(\mathcal{P})$ throughout this paper. We call this model passive if $\alpha = 0$ and active if $\alpha \neq 0$.
\end{defi}

\begin{rmrk}
    As $\activity$ has the opposite sign of $u_\text{act}$, contractile stress is modeled by $\activity > 0$ and extensile stress by $\activity < 0$.
\end{rmrk}

\section{Preliminaries}
\label{sec:preliminares}

\subsection{Notation}
We denote the real numbers by $\R$, the natural numbers by $\N := \{1, 2, \dots \}$ and set $\N_0 := \N \cup \{0\}$. In particular, we define $\R_{\geq 0} := [0, \infty)$ and $\R_{>0} := (0, \infty)$.
Moreover, we introduce the following scalar products: For $\vflow, \wflow \in \R^d$ and $\AMat, \BMat \in \R^{d \times d}$, we define $\vflow \cdot \wflow := \sum_{i=1}^d \vflow_i \wflow_i$ and $\AMat : \BMat := \sum_{i,j=1}^d \AMat_{ij} \BMat_{ij}$.
The trace of a tensor $\AMat \in \R^{d \times d}$ is given by $\tr{\AMat} := \sum_{i=1}^d \AMat_{ii}$.
The notation $\abs{\cdot}$ is used as follows: For a real number $a \in \R$, we set $\abs{a}$ to the absolute value of $a$, whereas $\abs{\vflow} := (\vflow \cdot \vflow)^{1/2}$ and $\abs{\AMat} := (\AMat : \AMat)^{1/2}$ for $\vflow \in \R^d$ and $\AMat \in \R^{d \times d}$.
Moreover, we set $\abs{U}$ to the $d$-dimensional Lebesgue measure of a measurable set $U \subset \R^d$.
For $\vflow, \wflow \in \R^d$, we define $\vflow \otimes \wflow := ( \vflow_i \wflow_j )_{1 \leq i,j \leq d}$.
We introduce the positive part $\pos{a}$ and the negative part $\negat{a}$ of a real number $a \in \R$ by $[a]_{\pm} := \frac{1}{2}(a \pm \abs{a})$.
For ease of notation, we write $\dH$ instead of $\text{d}\mathscr{H}^{d-1}$ in integrals for the Hausdorff measure.
The isotropic mollifier $\JeC: \L^1(\Omega) \to \W^{1, \infty}(\Omega)$ (cf. \eqref{def:mollifier}) is naturally extended to the tensor valued mollifier $\JeMatC: \LMat^1(\Omega) \to \WMat^{1, \infty}(\Omega)$.
The velocity gradient is defined by $(\nablax \uflow)_{1 \leq i,j \leq d} = (\partial_{\xflow_j} \uflow_i)$ for a differentiable function $\uflow : \Omega \to \R^d$. Moreover, we introduce the divergence of a differentiable function $\AMat : \Omega \to \R^{d \times d}$ by
\begin{align*}
\Divx(\AMat) := \lrcb{ \sum_{j=1}^d \partial_{\xflow_j} \AMat_{ij}}_{i=1}^d.
\end{align*}

\subsection{Function spaces}
We use the standard notation for Lebesgue and Sobolev spaces and define $\Lflow^p(\Omega) := \L^p(\Omega; \R^d)$, $\Wflow^{k, p}(\Omega) := \W^{k,p}(\Omega; \R^d)$, $\Hflow^k(\Omega) := \Wflow^{k, 2}($ $\Omega)$, $\LMat^p(\Omega) := \L^p(\Omega; \R^{d \times d})$, $\WMat^{k,p}(\Omega)$ $:=$ $\W^{k,p}(\Omega;$ $ \R^{d \times d})$, and $\HMat^k(\Omega) := \WMat^{k, 2}(\Omega)$ for $1 \leq p \leq \infty$ and $k \in \N$. Moreover, we introduce the spaces
\begin{align*}
\Hdiv &:= \left\{ \wflow \in \Hflow_0^1(\Omega):\ \Divx(\wflow) = 0 \right\}, \\
\Ldivp{2} &:= \left\{ \wflow \in \Lflow^2(\Omega):\ \intOm \wflow \cdot \nablax \varphi \dx = 0 \quad \forall \varphi \in \H^1(\Omega) \right\}, \\
\mathscr{V} &:= \left\{ \wflow \in \Cflow_0^\infty(\Omega):\ \Divx(\wflow) = 0 \right\}, \\
\Vk{k} &:= \overline{\mathscr{V}}^{\norm{\cdot}_{\Hflow_0^1(\Omega) \cap \Hflow^k(\Omega)}},\  k \in \N.
\end{align*}
We note that it holds that $\Vk{1} = \Hdiv$ (cf. Temam \cite[Chapter 1, Theorem 1.6]{temam1977navierNumerics}).
An introduction to Sobolev spaces $\W^{k,p}(\Sph)$ on the sphere can be found in \cite{kufner1977functionSpaces}. Furthermore, we define the dual spaces of $\W^{k,p}(\Omega \times \Sph)$ and $\Vk{k}$ by $(\W^{k,p}(\Omega \times \Sph))^\prime$ and $(\Vk{k})^\prime$, respectively. It will turn out useful to introduce the space
\begin{align*}
\Zp{p} := \left\{ \theta \in \LpOS{p}:\ \theta \geq 0 \text{ a.e. in } \OmSph \right\} \quad \forall 1 \leq p < \infty.
\end{align*}

\subsection{Frequently used formula}
Let us note some results which will be often referred to in future discussions.

\begin{lmm}
	Let the isotropic mollifier $\JeC: \L^1(\Omega) \to \W^{1, \infty}(\Omega)$ be defined by \eqref{def:mollifier}. Then, it holds that
	\begin{subequations}
		\begin{alignat}{2}
		\norm{ \Je{f} }_{\L^2(\Omega)} &\leq \norm{f}_{\L^2(\Omega)} \quad &&\forall f \in \L^2(\Omega), \label{mollifier:L2_stab}\\
		\intOm \Je{f} g \dx &= \intOm f \Je{g} \dx \quad &&\forall f,g \in \L^2(\Omega), \label{mollifier:fg}\\
		\norm{\Je{f}}_{\W^{1,\infty}(\Omega)} &\leq \CmollEps \norm{f}_{\L^1(\Omega)} \quad &&\forall f \in \L^1(\Omega). \label{mollifier:W_2_infty}
		\end{alignat}
	\end{subequations}
	Moreover, similar results hold for the tensor valued mollifier $\JeMatC$.
\end{lmm}

Furthermore, we will exploit the following formula on the unit sphere (cf. \cite[Lemma 2.1, Lemma 1.1]{chen_liu_2013_entropy_solution}):

\begin{lmm}
	Let $\vflow \in \Hflow^1(\Sph)$ and $f, g \in \H^1(\Sph)$. Then, it holds that
	\begin{subequations}
		\begin{align}
		\intSph \Divg(\vflow) f \dH &= - \intSph \vflow \cdot \nablag f \dH + (d - 1) \intSph (\vflow \cdot \mflow) f \dHm, \label{eq:pInt:divg} \\
		\intSph (\Deltag f) g \dH &= - \intSph \nablag f \cdot \nablag g \dH. \label{eq:pInt:nablag}
		\end{align}
	\end{subequations}
\end{lmm}

\begin{lmm}
	Let $f \in \W^{1,1}(\Sph)$ and $\AMat \in \Rdd$ be a constant matrix with $\tr{\AMat} = 0$. Then, it holds that
	\begin{align}
	\label{eq:chen_stress_formula}
	\intSph \lrcb{ (\IdMat - \mom) \AMat \mflow } \cdot \nablag f \dHm
	= \intSph f (d \mom - \IdMat) : \AMat \dHm.
	\end{align}
\end{lmm}

Moreover, the following lemma is frequently used, which follows directly from $\nablag = (\IdMat - \mom) \nablam$, $(\mflow \cdot \mflowtilde)^2 = (\mom) : (\mflowtilde \otimes \mflowtilde)$ for $\mflow, \mflowtilde \in \Sph$, and \eqref{def:pot}.
\begin{lmm}
	\label{lmm:nablag_g_pot_formula}
	It holds that
	\begin{align}
	\label{eq:nabla_g_pot_formula}
	\nablag \pot[\theta](\xflow, \mflow) = - 2 \potStr (\IdMat - \mom) \JeMat{ \intSph \theta(\cdot, \mflowtilde) \mflowtilde \otimes \mflowtilde \dHmt } (\xflow) \mflow
	\end{align}
	for all $\theta \in \LpOS{1}$ for $(\xflow, \mflow) \in \OmSph$.
\end{lmm}

\subsection{Formal entropy estimate}
\label{sec:formal_entropy_estimate}
Now, let us derive an entropy estimate on a formal level which exploits the energetic and entropic structure of the model $(\mathcal{P})$. 
First, assuming that $\psi > 0$, multiplying \eqref{eq:active:psi} by $\muPsi$, integrating over $\OmSph$, integrating by parts in $\Omega$ and on $\Sph$ (see \eqref{eq:pInt:divg} and \eqref{eq:pInt:nablag}), and noting \eqref{eq:active:divu}-\eqref{eq:active:bc:u}, \eqref{eq:active:bc:psi}, and \eqref{mollifier:fg}, we have that
\begin{align}
\label{formal:psi:muPsi}
\begin{split}
&\ddt \intOmSph \left[ \entrop(\psi) + \frac{1}{2} \psi \pot[\psi] \right] \dH\dx \\
&\qquad + \frac{\mollEps^2}{\DE} \intOmSph \psi \abs{\nablax \muPsi}^2 \dH\dx + \frac{1}{\DE} \intOmSph \psi \abs{\nablag \muPsi}^2 \dH\dx \\
&\quad = - \intOmSph \uflow \cdot \nablax \psi \pot[\psi] \dH\dx \\
&\qquad + \intOmSph \psi \lrcb{ (\IdMat - \mom) \nablax\uflow \mflow } \cdot \nablag \muPsi \dHm\dx.
\end{split}
\end{align}

Next, multiplying \eqref{eq:active:u} by $\uflow$, integrating over $\Omega$, noting \eqref{eq:active:divu} and \eqref{eq:active:bc:u} and some integration by parts yield
\begin{align*}
&\frac{1}{2} \ddt \intOm \abs{\uflow}^2 \dx + \frac{\viscFrac}{\RE} \intOm \abs{\nablax \uflow}^2 \dx \\
&\qquad + \frac{1 - \viscFrac}{2 \RE} \intOmSph \psi \lrcb{ \nablax \uflow : \mom }^2  \dHm\dx \\
&\quad= - \frac{1 - \viscFrac}{\RE\DE} \intOmSph \psi \lrcb{ (\IdMat - \mom) \nablax \uflow \mflow } \cdot \nablag \muPsi \dHm\dx \\
&\quad\quad + \frac{1 - \viscFrac}{\RE\DE} \intOmSph \uflow \cdot \nablax \psi \pot[\psi] \dH\dx \numberthis[formal:u:u] \\
&\quad\quad - \frac{\alpha (1 - \viscFrac)}{\RE \DE} \intOmSph \psi \nablax \uflow : \mom \dHm\dx,
\end{align*}
where we used that
\begin{align}
\label{formal:stress_formula}
\begin{split}
&\intOmSph \psi \nablax\uflow : \lrcb{ \left[ (\IdMat - \mom) \nablag \muPsi \right] \otimes \mflow } \dHm\dx \\
&\quad = \intOmSph \psi \lrcb{ (\IdMat - \mom) \nablax \uflow \mflow} \cdot \nablag \muPsi \dHm\dx.
\end{split}
\end{align}

Now, we infer from multiplying \eqref{formal:psi:muPsi} by $\frac{1-\viscFrac}{\RE\DE}$, combining \eqref{formal:psi:muPsi} - \eqref{formal:u:u}, and applying Young's inequality that
\begin{align*}
&\frac{1}{2} \ddt \intOm \abs{\uflow}^2 \dx + \frac{\viscFrac}{\RE} \intOm \abs{\nablax \uflow}^2 \dx \\
&\qquad + \frac{1 - \viscFrac}{2 \RE} \intOmSph \psi \lrcb{\nablax\uflow : \mom}^2 \dHm\dx \\
&\qquad + \frac{1 - \viscFrac}{\RE\DE} \ddt \intOmSph \left[ \entrop(\psi) + \frac{1}{2}\psi\pot[\psi] \right] \dH\dx \\
&\qquad + \frac{1 - \viscFrac}{\RE\DE} \lrcb{ \frac{\mollEps^2}{\DE} \intOmSph \psi \abs{\nablax \muPsi}^2 \dH\dx + \frac{1}{\DE} \intOmSph \psi \abs{\nablag \muPsi}^2 \dH\dx } \\
&\quad= - \frac{\alpha (1 - \viscFrac)}{\RE \DE} \intOmSph \psi \nablax \uflow : \mom \dHm\dx \numberthis[formal:summarize_1] \\
&\quad\leq \frac{1 - \viscFrac}{4\RE} \intOmSph \psi \lrcb{\nablax\uflow : \mom}^2 \dHm\dx + \frac{\alpha^2 (1 - \viscFrac)}{\RE \DE^2} \intOmSph \psi \dH\dx.
\end{align*}

Furthermore, to deal with the last term in \eqref{formal:summarize_1}, we obtain from \eqref{eq:active:divu}-\eqref{eq:active:bc:u}, \eqref{eq:active:bc:psi}, \eqref{eq:pInt:divg}, and integrating \eqref{eq:active:psi} over $\OmSph$ that
\begin{align}
\label{formal:mass_conserv}
\intOmSph \psi(\xflow, \mflow, t) \dHm\dx = \intOmSph \psi_0(\xflow, \mflow) \dHm\dx
\end{align}
for $t \in (0, T]$.
Therefore, we conclude from \eqref{formal:summarize_1} and \eqref{formal:mass_conserv} that
\begin{align*}
&\ddt \left[ \frac{1}{2} \intOm \abs{\uflow}^2 \dx + \frac{1 - \viscFrac}{\RE\DE} \intOmSph \left[ \entrop(\psi) + \frac{1}{2}\psi\pot[\psi] \right] \dH\dx \right] \\
&\quad + \frac{\viscFrac}{\RE} \intOm \abs{\nablax \uflow}^2 \dx + \frac{1 - \viscFrac}{4\RE} \intOmSph \psi \lrcb{\nablax\uflow : \mom}^2 \dHm\dx \\
&\quad + \frac{1 - \viscFrac}{\RE\DE} \lrcb{ \frac{\mollEps^2}{\DE} \intOmSph \psi \abs{\nablax \muPsi}^2 \dH\dx + \frac{1}{\DE} \intOmSph \psi \abs{\nablag \muPsi}^2 \dH\dx } \\
&\leq \frac{\alpha^2 (1 - \viscFrac)}{\RE \DE^2} \intOmSph \psi_0 \dH\dx. \numberthis[formal:entropy_est]
\end{align*}

We note that there is no need for an absorption argument in \eqref{formal:summarize_1} in the passive case as the right-hand sides vanishes for $\activity = 0$.


\section{Discrete-in-time approximations}
\label{sec:discrete_in_time_approx}
To make the formal computations in Section \ref{sec:formal_entropy_estimate} rigorous, we introduce the following regularization functions: For any $\regL > 1$, we define
\begin{subequations}
	\label{def:regularizations}
	\begin{align}
	\begin{split}
	\label{def:regL}
	\entropL(s) &:= 
	\begin{cases}
	\entrop(s) \equiv s (\log(s) - 1) + 1 \quad & \text{ for } 0 \leq s \leq \regL, \\
	\frac{s^2 - \regL^2}{2 \regL} + s(\log(\regL) - 1) + 1 \quad & \text{ for } \regL\leq s,
	\end{cases}
	\end{split} \\
	\begin{split}
	\label{def:QL}
	\QL(s) &:=
	\begin{cases}
	s \quad & \text{ for } s \leq \regL, \\
	\regL \quad & \text{ for } \regL \leq s,
	\end{cases}
	\end{split} \\
	\begin{split}
	\label{def:QZL}
	\QZL(s) &:= \max\{ 0,\ \QL(s) \} \quad \text{ for } s \in \R.
	\end{split}
	\end{align}
\end{subequations}

Furthermore, we note that the regularization functions satisfy the following properties (cf. Chen and Liu \cite{chen_liu_2013_entropy_solution} and Barrett and S\"uli \cite{barrett_sueli_2011_polymers_I, barrett_sueli_2011_polymers_II}):
\begin{subequations}
	\label{reg:props}
	\begin{alignat}{2}
	\QL \in \C^{0,1}(\R),\quad \entropL &\in \C^{2,1}(\R_{>0}) \cap \C(\R_{\geq 0}), \\
	\entropL & \text{ is convex on } [0, \infty), \quad && \label{reg:F_cvx} \\
	\entropL(s) & \geq \entrop(s) \quad && \forall s \geq 0, \label{reg:FL_geq_F} \\
	\entropL\lrcb{\QL(s) + \regDelta} & \leq \regDelta + \frac{\regDelta^2}{2} + \entrop(s + \regDelta) \quad && \forall \regDelta \in (0,1), \forall s \geq 0, \label{reg:FLQL_leq_FL} \\
	\ddentropL(s) &= \lrcb{\QL(s)}^{-1} \geq \frac{1}{s} \quad &&\forall s > 0, \label{reg:ddFL_geq} \\
	\ddentrop(s + \regDelta) & \leq \frac{1}{\regDelta} \quad && \forall \regDelta \in (0,1),\ \forall s \geq 0, \label{reg:ddFL_leq} \\
	\entropL(s) &\geq \frac{s^2}{4\regL} - C(\regL) \quad && \forall s \geq 0, \label{reg:FL_geq_ss} \\
	\abs{\QZL(s)} &\leq \abs{\QL(s)} \leq \abs{s} \quad && \forall s \in \R. \label{reg:QL_leq}
	\end{alignat}
\end{subequations}

\textit{\textbf{Assumptions on the initial data and the domain.}}
We assume the domain and the initial data to satisfy
\begin{subequations}
	\label{assumptions:domain_init_data}
	\begin{align}
	\begin{split}
	\label{eq:domain}
	& \Omega \subset \R^d \text{ is an open bounded Lipschitz domain},\quad d \in \{ 2, 3\},
	\end{split}\\
	\begin{split}
	\label{init:psi_0}
	& \psi_0 \in \L^2(\Omega; \L^1(\Sph)), \quad \psi_0 \geq 0 \text{ a.e. in } \OmSph  \\
	& \qquad \text{ with } \entrop(\psi_0) \in \LpOS{1},
	\end{split}\\
	\begin{split}
	\label{init:u_0}
	\uflow_0 \in \Ldivp{2}.
	\end{split}
	\end{align}
\end{subequations}

Let a time $T > 0$ and a number of time steps $N \in \N$ be given. Then, we define a time increment $\tau > 0$ by $\tau := \frac{T}{N}$ and $t_n := n \tau$ for $n \in \{0, \dots, N\}$. Moreover, we define the initial velocity $\uL{0} \in \Hdiv$ for our discrete-in-time scheme as the weak solution of
\begin{align}
\label{init:def:u_L_0}
\intOm \left[ \uL{0} \cdot \wflow + \regL^{-1} \nablax \uL{0} : \nablax \wflow \right] \dx = \intOm \uflow_0 \cdot \wflow \dx \quad \forall \wflow \in \Hdiv.
\end{align}
We deduce from \eqref{init:def:u_L_0} that
\begin{align}
\label{init:regularity_u_L_0}
\intOm \left[ \abs{\uL{0}}^2 + \regL^{-1} \abs{\nablax \uL{0}}^2 \right] \dx \leq \intOm \abs{\uflow_0}^2 \dx \leq C < \infty,
\end{align}
where $C > 0$ is independent of $\tau$ and $\regL$. Furthermore, we set
\begin{align}
\label{init:def:psi_L_0}
\psiL{0} := \QL(\psi_0) \in \Zp{2}.
\end{align}
In particular, we infer from \eqref{init:psi_0}, \eqref{init:def:psi_L_0}, and \eqref{reg:QL_leq} that
\begin{align}
\label{init:regularity_psi_L_0}
\intOmSph \psiL{0} \dH\dx \leq \intOmSph \psi_0 \dH\dx \leq C < \infty,
\end{align}
where $C>0$ is independent of $\tau$ and $\regL$.
Moreover, we obtain from the definition of the interaction potential \eqref{def:pot}, the nonnegativity of the Maier-Saupe interaction kernel \eqref{def:interaction_kernel}, \eqref{init:psi_0}, \eqref{init:def:psi_L_0}, \eqref{reg:QL_leq}, \eqref{mollifier:L2_stab}, and H\"older's inequality that
\begin{align*}
&\intOmSph \psiL{0} \pot[\psiL{0}] \dH\dx \\
&\quad\leq C(\potStr, d) \intOm \lrcb{ \intSph \psi_0(\xflow, \mflow) \dHm }^2 \dx 
\leq C < \infty,	\numberthis[init:psi_0_pot_0]
\end{align*}
where $C > 0$ is independent of $\tau$ and $\regL$.

Our discrete-in-time problem $\PLtau$ reads as follows:\\
\hspace*{3ex} $\PLtau$ Let $\uL{0} \in \Hdiv$ and $\psiL{0} \in \Zp{2}$ be given by \eqref{init:def:u_L_0} and \eqref{init:def:psi_L_0}, respectively. Then, for $n \in \{1, \dots, N\}$, given $(\uL{n-1}, \psiL{n-1})$ $\in$ $\Hdiv \times \Zp{2}$, find $(\uL{n}, \psiL{n})$ $\in$ $\Hdiv \times (\HkOS{1} \cap \Zp{2})$ such that, for all $\wflow \in \Hdiv$ and $\theta \in \H^1(\Omega \times \Sph)$,
\begin{subequations}
	\label{def:PLtau}
	\begin{align}
	\begin{split}
	\label{eq:PLtau:u}
	& \RE \DE \intOm \frac{\uL{n} - \uL{n-1}}{\tau} \cdot \wflow \dx + \RE \DE \intOm \lrcb{ (\uL{n-1} \cdot \nablax) \uL{n} } \cdot \wflow \dx \\
	&\qquad + \viscFrac \DE \intOm \nablax \uL{n} : \nablax \wflow \dx \\
	&\qquad + \frac{(1 - \viscFrac) \DE}{2} \intOmSph \QZL(\psiL{n}) (\nablax \uL{n} : \mom) (\nablax \wflow : \mom) \dHm \dx \\
	&\quad = - (1 - \viscFrac) \intOmSph \psiL{n} \nablax \potL{n} \cdot \wflow \dH\dx  \\
	&\quad \quad - (1 - \viscFrac) \intOmSph \QZL(\psiL{n}) \lrcb{ (\IdMat - \mom) \nablax \wflow \mflow } \cdot \nablag \potL{n} \dHm\dx \\
	&\quad \quad - (1 - \viscFrac) \intOmSph \psiL{n} (d\mom - \IdMat) : \nablax \wflow \dHm\dx \\
	&\quad \quad - \activity(1 - \viscFrac) \intOmSph \QZL(\psiL{n}) \mom : \nablax \wflow \dHm\dx, 
	\end{split} \\
	\begin{split}
	\label{eq:PLtau:psi}
	&\intOmSph \frac{\psiL{n} - \psiL{n-1}}{\tau} \theta \dH\dx - \intOmSph \psiL{n} \uL{n} \cdot \nablax \theta \dH\dx \\
	&\quad + \frac{1}{\DE} \intOmSph \left[ \mollEps^2 \nablax \psiL{n} \cdot \nablax \theta + \nablag \psiL{n} \cdot \nablag \theta \right] \dH\dx \\
	&\quad + \frac{1}{\DE} \intOmSph \left[ \mollEps^2 \QZL(\psiL{n}) \nablax \potL{n} \cdot \nablax \theta + \QZL(\psiL{n}) \nablag \potL{n} \cdot \nablag \theta \right] \dH\dx \\
	&\quad - \intOmSph \QZL(\psiL{n}) \lrcb{(\IdMat - \mom) \nablax \uL{n} \mflow} \cdot \nablag \theta \dHm\dx = 0.
	\end{split}
	\end{align}
\end{subequations}

\begin{rmrk}
	\begin{enumerate}
		\item[i)] On noting \eqref{formal:stress_formula} and \eqref{eq:chen_stress_formula}, the elastic stress tensor in the Navier-Stokes equation can be rewritten as
		\begin{align*}
		& \intOmSph \Divx \lrcb{ \psi \left[ (\IdMat - \mom) \nablag \muPsi \right] \otimes \mflow } \cdot \wflow \dHm\dx \\
		&\quad = - \intOmSph \psi \lrcb{ (\IdMat - \mom) \nablax \wflow \mflow } \cdot \nablag \muPsi  \dHm\dx \\
		&\quad = - \intOmSph \psi (d \mom - \IdMat) : \nablax \wflow \dHm\dx \numberthis[rmrk:stress_tensor] \\
		&\quad \qquad - \intOmSph \psi \lrcb{ (\IdMat - \mom) \nablax \wflow \mflow } \cdot \nablag \pot[\psi] \dHm\dx
		\end{align*}
		for all $\wflow \in \Hdiv$.
		\item[ii)] For reasons of readability, we use the abbreviation
		\begin{align}
		\label{def:potL}
		\potL{n} := \frac{1}{2} \pot\left[ \psiL{n} + \psiL{n-1} \right] \qquad \text{ for } n \in \{1, \dots, N\}.
		\end{align}
	\end{enumerate}
\end{rmrk}

\subsection{Existence of solutions to \texorpdfstring{$\PLtau$}{}}
\label{sec:ex_of_sols_to_PLtau}
In this section, our goal is to prove the existence of a solution $(\uL{n}, \psiL{n})$ $\in$ $\Hdiv \times (\HkOS{1} \cap \Zp{2})$ to $\PLtau$ by applying Schaefer's fixed-point theorem (cf. Evans \cite{evans2010pdes}).
First of all, we prove a lemma which will be used very frequently.

\begin{lmm}
	Let $\theta \in \L^1(\OmSph)$. Then, it holds that
	\begin{align}
	\label{est:nabla_pot_Linfty}
	\norm{\pot[\theta]}_{\W^{1, \infty}(\OmSph)}
	\leq C(\CmollEps, \potStr, d) \norm{\theta}_{\L^1(\OmSph)},
	\end{align}
	where $\CmollEps > 0$ is the constant in \eqref{mollifier:W_2_infty}.
\end{lmm}

\begin{proof}
	On noting the definition of the interaction potential \eqref{def:pot} and $(\mflow \cdot \mflowtilde)^2 = \mom : \mflowtilde \otimes \mflowtilde$ for $\mflow, \mflowtilde \in \Sph$, we have that
	\begin{align*}
		\pot[\theta](\xflow, \mflow)
		&= \potStr \intSph \Je{\theta(\cdot, \mflowtilde)}(\xflow) \lrcb{1 - (\mflow \cdot \mflowtilde)^2} \dHmt \\
		&= \potStr \Je{ \intSph \theta(\cdot, \mflowtilde) \dHmt }(\xflow) \\
		&\quad - \potStr \JeMat{ \intSph \theta(\cdot, \mflowtilde) \mflowtilde \otimes \mflowtilde \dHmt }(\xflow) : (\mom)
	\end{align*}
	for $(\xflow, \mflow) \in \OmSph$. Hence, we infer from \eqref{mollifier:W_2_infty} that
	\begin{align*}
		&\norm{\pot[\theta]}_{\LpOS{\infty}} + \norm{\nablax \pot[\theta]}_{\L^\infty(\OmSph)} \\
		&\quad\leq C(\CmollEps, \potStr) \norm{\theta}_{\LpOS{1}}  + C(\CmollEps, \potStr, d) \intOmSph \abs{\theta(\xflow, \mflowtilde)} \abs{\mflowtilde \otimes \mflowtilde} \dHmt\dx \\
		&\quad\leq C(\CmollEps, \potStr, d) \norm{\theta}_{\LpOS{1}}.
	\end{align*}
	Furthermore, we infer from \eqref{mollifier:W_2_infty} and \eqref{eq:nabla_g_pot_formula} that
	\begin{align*}
	\norm{ \nablag \pot[\theta] }_{\L^\infty(\OmSph)} 
	&\quad\leq C(d) \potStr \esssup_{\xflow \in \Omega} \abs{ \JeMat{ \intSph \theta(\cdot, \mflowtilde) \mflowtilde \otimes \mflowtilde \dHmt } (\xflow) } \\
	&\quad\leq C(\CmollEps, \potStr, d) \norm{\theta}_{\L^1(\OmSph)}.
	\end{align*}
\end{proof}

\textit{\textbf{Fixed-point iteration.}} To introduce our fixed-point scheme, we rewrite \eqref{eq:PLtau:u} as
\begin{align}
\label{fp:eq:a}
a(\psiL{n})(\uL{n}, \wflow) = k(\psiL{n})(\wflow) \quad \forall \wflow \in \Hdiv,
\end{align}
where
\begin{align}
\label{fp:def:a}
\begin{split}
&a(\psiov)(\uflow, \wflow)\\
&\quad:= \RE \DE \intOm \uflow \cdot \wflow \dx + \tau \RE \DE \intOm \lrcb{ (\uL{n-1} \cdot \nablax) \uflow } \cdot \wflow \dx \\
&\quad\quad + \tau \viscFrac \DE \intOm \nablax \uflow : \nablax \wflow \dx \\
&\quad\quad + \tau \frac{(1-\viscFrac) \DE}{2} \intOmSph \QZL(\psiov) (\nablax \uflow : \mom) (\nablax \wflow : \mom) \dHm \dx
\end{split}
\end{align}
and
\begin{align}
\label{fp:def:k}
\begin{split}
&k(\psiov)(\wflow)\\
&\quad:= \RE \DE \intOm \uL{n-1} \cdot \wflow \dx - \tau \frac{1 - \viscFrac}{2} \intOmSph \psiov \nablax \pot\left[\psiov + \psiL{n-1}\right] \cdot \wflow \dH\dx \\
&\quad\quad - \tau \frac{1 - \viscFrac}{2} \intOmSph \QZL(\psiov) \lrcb{ (\IdMat - \mom) \nablax \wflow \mflow } \cdot \nablag \pot\left[\psiov + \psiL{n-1}\right] \dHm\dx \\
&\quad\quad - \tau (1 - \viscFrac) \intOmSph \psiov (d\mom - \IdMat) : \nablax \wflow \dHm\dx \\
&\quad\quad - \tau \activity (1 - \viscFrac) \intOmSph \QZL(\psiov) \mom : \nablax \wflow \dHm\dx
\end{split}
\end{align}
for all $\uflow, \wflow \in \Hdiv$ and for all $\psiov \in \L^2(\OmSph)$.
Moreover, we rewrite \eqref{eq:PLtau:psi} as
\begin{align}
\label{fp:eq:b}
\begin{split}
b(\uL{n})(\psiL{n}, \theta) = l(\uL{n}, \psiL{n})(\theta) \qquad \forall \theta \in \HkOS{1},
\end{split}
\end{align}
where
\begin{align}
\label{fp:def:b}
\begin{split}
b(\uflow)(\psi, \theta)
&:= \intOmSph \psi \theta \dH\dx - \tau \intOmSph \psi \uflow \cdot \nablax \theta \dH\dx \\
&\quad + \frac{\tau}{\DE} \intOmSph \left[ \mollEps^2 \nablax \psi \cdot \nablax \theta + \nablag \psi \cdot \nablag \theta \right] \dH\dx
\end{split}
\end{align}
and
\begin{align}
\label{fp:def:l}
\begin{split}
l(\uflow, \psiov)(\theta)
&:= \intOmSph \psiL{n-1} \theta \dH\dx \\
&\quad - \frac{\tau}{2 \DE} \intOmSph \mollEps^2 \QZL(\psiov) \nablax \pot\left[\psiov + \psiL{n-1}\right] \cdot \nablax \theta \dH\dx \\
&\quad - \frac{\tau}{2 \DE} \intOmSph \QZL(\psiov) \nablag \pot\left[\psiov + \psiL{n-1}\right] \cdot \nablag \theta \dH\dx \\
&\quad + \tau \intOmSph \QZL(\psiov) \lrcb{ (\IdMat - \mom) \nablax \uflow \mflow } \cdot \nablag \theta \dHm\dx
\end{split}
\end{align}
for all $\psi, \theta \in \HkOS{1}$ and for all $\psiov \in \LpOS{2}$, $\uflow \in \Hdiv$.
\begin{lmm}
	\label{lmm:fixed_point_iter}
	Suppose that the assumptions \eqref{assumptions:domain_init_data} hold and let $\psiov \in \L^2(\OmSph)$.
	\begin{enumerate}
		\item[i)] There exists a unique $\uflow \in \Hdiv$ such that
		\begin{align}
		\label{fp:iteration:i}
		a(\psiov)(\uflow, \wflow) = k(\psiov)(\wflow) \quad \forall \wflow \in \Hdiv.
		\end{align}
		\item[ii)] Now, let $\uflow \in \Hdiv$ be the unique solution to \eqref{fp:iteration:i}. Then, there exists a unique $\psi$ $\in$ $\HkOS{1}$ such that
		\begin{align}
		\label{fp:iteration:ii}
		b(\uflow)(\psi, \theta) = l(\uflow, \psiov)(\theta) \quad \forall \theta \in \HkOS{1}.
		\end{align}
		\item[iii)] Moreover, solutions to \eqref{fp:iteration:i} and \eqref{fp:iteration:ii} satisfy
		\begin{align}
		\label{fp:iteration:iii:u}
		\begin{split}
			&\DE\, \min\{ \RE, \tau\viscFrac \} \norm{\uflow}_{\H^1(\Omega)} \\
			&\quad \leq C \lrcb{ 1 + \tau \regL + \tau \regL \norm{\psiov}_{\LpOS{2}} + \tau \norm{\psiov}_{\LpOS{2}}^2 },
		\end{split}
		\end{align}
		where $C = C(\RE, \DE, \uL{n-1}, \CmollEps, \potStr, d, \viscFrac, \activity, \psiL{n-1}) > 0$ and
		\begin{align}
		\label{fp:iteration:iii:psi}
		\begin{split}
			&\min\left\{ 1, \frac{\tau}{\DE}, \frac{\tau \mollEps^2}{\DE} \right\} \norm{\psi}_{\HkOS{1}} \\
			&\quad \leq C \lrcb{ 1 + \tau \regL + \tau \regL \norm{\psiov}_{\LpOS{1}} + \tau \regL \norm{\nablax \uflow}_{\L^2(\Omega)} },
		\end{split}
		\end{align}
		where $C = C(\psiL{n-1}, \CmollEps, \potStr, d, \mollEps, \DE^{-1}) > 0$.
	\end{enumerate}
\end{lmm}

\begin{defi}[Fixed-point iteration]
	\label{def:fixed_point_iteration}
	We define the operator $\Xi: \LpOS{2} \to \LpOS{2}$ as follows: For given $\psiov \in \LpOS{2}$, we introduce $\Xi[\psiov] \in \H^1(\Omega \times \Sph)$ as the solution of the variational problem \eqref{fp:iteration:ii}, where $\uflow \in \Hdiv$ is the solution of the variational problem \eqref{fp:iteration:i}.
\end{defi}

\begin{rmrk}
	We infer from Lemma \ref{lmm:fixed_point_iter} that the operator $\Xi$ defined in Definition \ref{def:fixed_point_iteration} is well defined.
	Clearly, if $\psi \in \HkOS{1}$ is a fixed-point of $\Xi$, then the tuple $(\uL{n}, \psiL{n})$ $:=$ $(\uflow, \psi)$ is a solution to \eqref{def:PLtau}. Hence, proving the existence of a fixed-point of $\Xi$ is equivalent to proving the existence of a solution to \eqref{def:PLtau}.
\end{rmrk}


\begin{proof}[Proof of Lemma \ref{lmm:fixed_point_iter}]
	Let $\psiov \in \L^2(\OmSph)$.\\
	\textit{Proof of i).}
	For all $\uflow \in \Hdiv$, on noting
	\begin{align*}
	\intOm \lrcb{ (\uL{n-1} \cdot \nablax) \uflow } \cdot \uflow \dx = 0,
	\end{align*}
	we have that
	\begin{align}
	\label{fp:a_coercive}
	\begin{split}
	a(\psiov)(\uflow, \uflow)
    \geq \DE \, \min\{ \RE, \tau \viscFrac \} \norm{\uflow}_{\H^1(\Omega)}^2.
	\end{split}
	\end{align}
	
	Moreover, we infer from $\H^1(\Omega) \hookrightarrow \Lflow^6(\Omega)$ that, for all $\uflow, \wflow \in \Hdiv$,
	\begin{align}
	\label{fp:a_bounded}
	\begin{split}
	\abs{a(\psiov)(\uflow, \wflow)}
	\leq C\lrcb{ \tau, \RE, \DE, \gamma, d, \uL{n-1}, \regL } \norm{\uflow}_{\H^1(\Omega)} \norm{\wflow}_{\H^1(\Omega)}.
	\end{split}
	\end{align}
	
	Furthermore, we obtain from \eqref{reg:QL_leq} and \eqref{est:nabla_pot_Linfty} that
	\begin{align}
	\label{fp:k_bounded}
	\begin{split}
	\abs{k(\psiov)(\wflow)}
    \leq C \lrcb{1 + \tau \regL + \tau \norm{\psiov}_{\LpOS{2}}^2 + \tau\regL \norm{\psiov}_{\LpOS{2}}} \norm{\wflow}_{\H^1(\Omega)},
	\end{split}
	\end{align}
	where $C = C \lrcb{ \RE, \DE, \uL{n-1}, \CmollEps, \potStr, d, \viscFrac, \activity, \psiL{n-1} } > 0$ and $\CmollEps > 0$ is the constant in \eqref{mollifier:W_2_infty}.
	We conclude from \eqref{fp:a_coercive}-\eqref{fp:a_bounded} that $a(\psiov)(\cdot, \cdot)$ is a bounded, coercive bilinear form on $\Hdiv \times \Hdiv$ and from \eqref{fp:k_bounded} that $k(\psiov)(\cdot) \in (\Hdiv)^\prime$. Hence, we obtain from Lax-Milgram's theorem the existence of a unique solution $\uflow \in \Hdiv$ to \eqref{fp:iteration:i}.

	\textit{Proof of ii).} Let $\uflow \in \Hdiv$ be the unique solution to \eqref{fp:iteration:i}. For all $\psi \in \H^1(\Omega \times \Sph)$, 
	on noting
	\begin{align*}
	\intOmSph \psi \uflow \cdot \nablax \psi \dH\dx = \intOmSph \frac{1}{2} \uflow \cdot \nablax (\psi^2) \dH\dx = 0,
	\end{align*}
	we have that
	\begin{align}
	\label{fp:b_coercive}
	\begin{split}
	b(\uflow)(\psi, \psi)
	\geq \min \left\{ 1, \frac{\tau}{\DE}, \frac{\tau}{\DE}\mollEps^2 \right\} \norm{\psi}_{\HkOS{1}}^2.
	\end{split}
	\end{align}
	
	Furthermore, we infer from $\HkOS{1} \hookrightarrow \LpOS{3}$ that
	\begin{align}
	\label{fp:b_bounded}
	\begin{split}
	\abs{b(\uflow)(\psi, \theta)}
	\leq \lrcb{ 1 + \frac{\tau}{\DE} + \frac{\tau}{\DE}\mollEps^2 + C(d) \tau \norm{\uflow}_{\L^6(\Omega)} } \norm{\psi}_{\HkOS{1}} \norm{\theta}_{\HkOS{1}}
	\end{split}
	\end{align}
	for all $\psi, \theta \in$ $\H^1(\Omega$ $\times\Sph)$.
	Moreover, we obtain from \eqref{est:nabla_pot_Linfty} that
	\begin{align}
	\label{fp:l_bounded}
	\begin{split}
	\abs{l(\uflow, \psiov)(\theta)}
	\leq C \Big( 1 + \tau \regL + \tau \regL \norm{\psiov}_{\LpOS{1}} + \tau \regL \norm{\nablax \uflow}_{\L^2(\Omega)} \Big) \norm{\theta}_{\HkOS{1}},
	\end{split}
	\end{align}
	where $C = C(\psiL{n-1}, \CmollEps, \potStr, d, \mollEps, \DE^{-1}) > 0$.
	Consequently, on noting \eqref{fp:b_coercive}-\eqref{fp:l_bounded}, it follows that $b(\uflow)(\cdot, \cdot)$ is a bounded, coercive bilinear form on $\H^1(\Omega \times \Sph)$ $\times$ $\HkOS{1}$ and that $l(\uflow, \psiov)(\cdot)$ $\in$ $\lrcb{\HkOS{1}}^\prime$. Hence, applying Lax-Milgram's theorem completes the proof of \textit{ii)}.

	\textit{Proof of iii).}
	On noting \eqref{fp:a_coercive}, \eqref{fp:k_bounded}, \eqref{fp:b_coercive}, and \eqref{fp:l_bounded}, it follows immediately that \eqref{fp:iteration:iii:u} and \eqref{fp:iteration:iii:psi} hold.
\end{proof}

As proving the existence of a fixed-point of $\Xi$ is equivalent to proving the existence of a solution to \eqref{def:PLtau}, let us show the following lemma:
\begin{lmm}[Existence of a fixed-point]
	\label{lmm:existence_fp}
	Suppose that the assumptions \eqref{assumptions:domain_init_data} hold.
	Then, the operator $\Xi: \LpOS{2} \to \LpOS{2}$ has a fixed-point.
\end{lmm}

\begin{proof}
	The idea to prove this lemma is applying Schaefer's fixed-point theorem (cf. Evans \cite{evans2010pdes}).
	Before doing so however we need to verify the following three conditions:
	\begin{enumerate}
		\item[i)] $\Xi: \LpOS{2} \to \LpOS{2}$ is continuous.
		\item[ii)] $\Xi: \LpOS{2} \to \LpOS{2}$ is compact.
		\item[iii)] The set
		\begin{align*}
		\left\{ \psi \in \LpOS{2}:\ \psi = \lambda \Xi(\psi) \text{ for some } 0 < \lambda \leq 1 \right\}
		\end{align*}
		is bounded in $\LpOS{2}$.
	\end{enumerate}
	
	\textit{Proof of i).} Let $\psi, \psi_k \in \LpOS{2}$, $k \in \N$, such that
	\begin{subequations}
		\begin{align}
		\begin{split}
		\label{fp:ex:i:cvg:psiov_k}
		\psi_k \to \psi \quad \text{ in } \LpOS{2} \text{ as } k \to \infty.
		\end{split}
		\end{align}
	\end{subequations}
	
	Our goal is to show that
	\begin{align}
	\label{fp:ex:i:goal}
	\Xi[\psi_k] \to \Xi[\psi] \quad \text{ in } \LpOS{2} \text{ as } k \to \infty.
	\end{align}
	
	First, we infer from Lemma \ref{lmm:fixed_point_iter} and the definition of $\Xi$ that there exist unique $\uflow, \uflow_k$ $\in$ $\Hdiv$, $k \in \N$, such that
	\begin{subequations}
		\begin{alignat}{2}
		a(\psi)(\uflow, \wflow) &= k(\psi)(\wflow) \quad && \forall \wflow \in \Hdiv, \label{fp:ex:i:a_u}\\
		a(\psi_k)(\uflow_k, \wflow) &= k(\psi_k)(\wflow) \quad && \forall \wflow \in \Hdiv,\ \forall k \in \N, \label{fp:ex:i:a_uk}\\
		b(\uflow)(\Xi[\psi], \theta) &= l(\uflow, \psi)(\theta) \quad && \forall \theta \in \HkOS{1}, \label{fp:ex:i:b_psi}\\
		b(\uflow_k)(\Xi[\psi_k], \theta) &= l(\uflow_k, \psi_k)(\theta) \quad && \forall \theta \in \HkOS{1},\ \forall k \in \N. \label{fp:ex:i:b_psi_k}
		\end{alignat}
	\end{subequations}
	
	Now, let us derive uniform bounds on $\{\Xi[\psi_k]\}_{k \in \N}$ in $\HkOS{1}$ and on $\{ \uflow_k \}_{k \in \N}$ in $\Hdiv$ to extract converging subsequences in appropriate spaces.
	Taking $\theta := \Xi[\psi_k]$ $\in$ $\HkOS{1}$ as a test function in \eqref{fp:ex:i:b_psi_k} and applying \eqref{fp:iteration:iii:psi}, we have that
	\begin{align}
	\label{fp:ex:i:psi_k_bound_2}
	\norm{\Xi[\psi_k]}_{\HkOS{1}}
	\leq C \lrcb{ 1 + \norm{\psi_k}_{\LpOS{1}} + \norm{\nablax \uflow_k}_{\L^2(\Omega)} },
	\end{align}
	where $C = C\lrcb{ \psiL{n-1}, \CmollEps, \potStr, d, \regL, \tau, \tau^{-1}, \DE, \DE^{-1}, \mollEps, \mollEps^{-1} } > 0$.
	So far, we do not control the right-hand side of \eqref{fp:ex:i:psi_k_bound_2}.
	Consequently, it remains to derive uniform bounds on $\{ \uflow_k \}_{k \in \N}$ in $\Hdiv$.
	Taking $\wflow := \uflow_k$ $\in$ $\Hdiv$ as a test function in \eqref{fp:ex:i:a_uk}, we deduce from \eqref{fp:iteration:iii:u} that
	\begin{align}
	\label{fp:ex:i:uk_bound_1}
	\begin{split}
	&\DE \, \min \left\{ \RE, \tau \viscFrac \right\} \norm{\uflow_k}_{\H^1(\Omega)}^2 \\
	&\quad \leq C \lrcb{1 + \norm{\psi_k}_{\LpOS{2}} + \norm{\psi_k}_{\LpOS{2}}^2} \norm{\uflow_k}_{\H^1(\Omega)},
	\end{split}
	\end{align}
	where $C = C(\RE, \DE, \uL{n-1}, \CmollEps, \potStr, d, \tau, \viscFrac, \regL, \activity, \psiL{n-1}) > 0$.
	Moreover, we obtain from \eqref{fp:ex:i:cvg:psiov_k} that there exists a constant $C > 0$ independent of $k$ such that
	\begin{align}
	\label{fp:ex:i:psiovk_bound}
	\norm{\psi_k}_{\LpOS{2}} \leq C \quad \forall k \in \N.
	\end{align}
	Therefore, we deduce from \eqref{fp:ex:i:uk_bound_1} and \eqref{fp:ex:i:psiovk_bound} that
	\begin{align}
	\label{fp:ex:i:uk_bound_2}
	\norm{\uflow_k}_{\H^1(\Omega)}
	\leq C,
	\end{align}
	where $C = C\lrcb{ \RE, \RE^{-1}, \DE, \DE^{-1}, \viscFrac, \viscFrac^{-1}, \tau, \tau^{-1}, \uL{n-1}, \CmollEps, \potStr, d, \regL, \activity, \psiL{n-1} } > 0$ is independent of $k$.
	Hence, inserting \eqref{fp:ex:i:psiovk_bound} and \eqref{fp:ex:i:uk_bound_2} in \eqref{fp:ex:i:psi_k_bound_2} yields
	\begin{align}
	\label{fp:ex:i:psi_k_bound_3}
	\norm{\Xi[\psi_k]}_{\HkOS{1}}
	\leq  C,
	\end{align}
	where $C = C\lrcb{ \uL{n-1}, \psiL{n-1}, \RE, \RE^{-1}, \DE, \DE^{-1}, \tau, \tau^{-1}, \viscFrac, \viscFrac^{-1}, \mollEps, \mollEps^{-1}, \CmollEps, \regL, \activity, \potStr, d } > 0$ is independent of $k$.
	Therefore, on noting \eqref{fp:ex:i:uk_bound_2}, \eqref{fp:ex:i:psi_k_bound_3}, the Lipschitz continuity of $\QZL$, and the compact embeddings 
	\begin{align*}
	\HkOS{1} \hookrightarrow \hookrightarrow \LpOS{2} \quad \text{ and } \quad \Hdiv \hookrightarrow \hookrightarrow \Lflow^2(\Omega),
	\end{align*}
	we deduce that there exists a subsequence $\{ \uflow_{k_i}, \Xi[\psi_{k_i}] \}_{i \in \N}$ of $\{ \uflow_k, \Xi[\psi_k] \}_{k \in \N}$ and functions $\uflowhat \in \Hdiv$ and $\psihat \in \HkOS{1}$ such that
	\begin{subequations}
		\begin{alignat}{2}
		\Xi[\psi_{k_i}] &\to \psihat \quad &&\text{ strongly in } \LpOS{2}, \label{fp:ex:i:cvg:psi_ki_strong} \\
		\Xi[\psi_{k_i}] &\weak \psihat \quad &&\text{ weakly in } \HkOS{1}, \label{fp:ex:i:cvg:psi_ki_weak}\\
		\QZL(\psi_{k_i}) &\to \QZL(\psi) \quad &&\text{ strongly in } \LpOS{2}, \label{fp:ex:i:cvg:QZL_psiov_ki}\\
		\uflow_{k_i} & \to \uflowhat \quad &&\text{ strongly in } \Lflow^2(\Omega), \label{fp:ex:i:cvg:u_ki_strong}\\
		\uflow_{k_i} & \weak \uflowhat \quad &&\text{ weakly in } \Hdiv \label{fp:ex:i:cvg:u_ki_weak}
		\end{alignat}
	\end{subequations}
	as $i \to \infty$.
	In particular, we infer from \eqref{est:nabla_pot_Linfty} and $\psi_{k_i} \to \psi$ strongly in $\LpOS{2}$ that
	\begin{align}
	\label{fp:ex:i:cvg:pot_psiov_ki}
	\pot[\psi_{k_i}] \to \pot[\psi] \quad \text{ strongly in } \W^{1, \infty}(\OmSph) \text{ as } i \to \infty.
	\end{align}
	Moreover, we obtain from \eqref{fp:ex:i:a_uk} that
	\begin{align}
	\label{fp:ex:i:a_uki}
	a(\psi_{k_i})(\uflow_{k_i}, \wflow) = k(\psi_{k_i})(\wflow) \quad \forall \wflow \in \Hdiv,\ \forall i \in \N.
	\end{align}
	On noting \eqref{fp:ex:i:cvg:psi_ki_weak}, \eqref{fp:ex:i:cvg:QZL_psiov_ki}, \eqref{fp:ex:i:cvg:u_ki_strong}, \eqref{fp:ex:i:cvg:u_ki_weak}, and \eqref{fp:ex:i:cvg:pot_psiov_ki}, we pass to the limit in \eqref{fp:ex:i:a_uki} as $i \to \infty$ to deduce that
	\begin{align}
	\label{fp:ex:i:a_uhat}
	a(\psi)(\uflowhat, \wflow) = k(\psi)(\wflow) \quad \forall \wflow \in \Hdiv.
	\end{align}
	
	Furthermore, it follows from \eqref{fp:ex:i:b_psi_k} that
	\begin{align}
	\label{fp:ex:i:b_psi_ki}
	b(\uflow_{k_i})(\Xi[\psi_{k_i}], \theta) = l(\uflow_{k_i}, \psi_{k_i})(\theta) \quad \forall \theta \in \HkOS{1},\ \forall i \in \N.
	\end{align}
	On noting \eqref{fp:ex:i:cvg:psi_ki_weak}, \eqref{fp:ex:i:cvg:QZL_psiov_ki}, \eqref{fp:ex:i:cvg:u_ki_strong}, \eqref{fp:ex:i:cvg:u_ki_weak}, and \eqref{fp:ex:i:cvg:pot_psiov_ki}, we pass to the limit in \eqref{fp:ex:i:b_psi_ki} as $i \to \infty$ to deduce that
	\begin{align}
	\label{fp:ex:i:b_psihat}
	b(\uflowhat)(\psihat, \theta) = l(\uflowhat, \psi)(\theta) \quad \forall \theta \in \HkOS{1}.
	\end{align}
	Now, we infer from the uniqueness of solutions to \eqref{fp:iteration:i} and \eqref{fp:iteration:ii} and from the comparison of \eqref{fp:ex:i:a_u} and \eqref{fp:ex:i:a_uhat} and of \eqref{fp:ex:i:b_psi} and \eqref{fp:ex:i:b_psihat} that
	\begin{align}
	\label{fp:ex:i:limit_id}
	\uflow = \uflowhat \quad \text{ and } \quad \Xi[\psi] = \psihat.
	\end{align}
	Consequently, on noting \eqref{fp:ex:i:cvg:psi_ki_strong} and \eqref{fp:ex:i:limit_id}, we have that
	\begin{align}
	\label{fp:ex:i:psi_ki_limit}
	\Xi[\psi_{k_i}] \to \psihat = \Xi[\psi] \quad \text{ in } \LpOS{2} \text{ as } i \to \infty.
	\end{align}
	Therefore, we infer from the uniqueness of solutions to \eqref{fp:iteration:i} and \eqref{fp:iteration:ii} that we can find for any subsequence $\{ \Xi[\psi_{k_j}] \}_{j \in \N}$ of $\{ \Xi[\psi_k] \}_{k \in \N}$ another subsequence $\{ \Xi[\psi_{k_{j_i}}] \}_{i \in \N}$ of $\{ \Xi[\psi_{k_j]} \}_{j \in \N}$ such that
	\begin{align}
	\label{fp:ex:i:cvg:psi_kji}
	\Xi[\psi_{k_{j_i}}] \to \Xi[\psi] \quad \text{ in } \LpOS{2} \text{ as } i \to \infty.
	\end{align}
	Hence, we can return to the original series $\{ \Xi[\psi_k] \}_{k \in \N}$ in \eqref{fp:ex:i:cvg:psi_kji}, which proves \eqref{fp:ex:i:goal}.

	\textit{Proof of ii).}
	It follows directly from \eqref{fp:iteration:iii:u}, \eqref{fp:iteration:iii:psi}, and the compact embedding $\H^1(\Omega \times \Sph)$ $\hookrightarrow\hookrightarrow$ $\L^2(\Omega \times \Sph)$ that \textit{ii)} holds.

	\textit{Proof of iii).}
	Let $\lambda \in (0,1]$ and $\psi \in \L^2(\OmSph)$ such that
	\begin{align}
	\label{fp:ex:iii:def_fp}
	\psi = \lambda \Xi(\psi).
	\end{align}
	We infer from Lemma \ref{lmm:fixed_point_iter} and the definition of $\Xi$ that there exists a unique $\uflow \in \Hdiv$ such that
	\begin{subequations}
		\begin{alignat}{2}
		a(\psi)(\uflow, \wflow) &= k(\psi)(\wflow) \quad && \forall \wflow \in \Hdiv, \label{fp:ex:iii:eq:a} \\
		b(\uflow)(\psi, \theta) &= \lambda l(\uflow, \psi)(\theta) \quad && \forall \theta \in \HkOS{1}. \label{fp:ex:iii:eq:b}
		\end{alignat}
	\end{subequations}
	In particular, it holds that $\psi \in \HkOS{1}$.
	Moreover, on choosing $\theta := \negat{\psi} \in \HkOS{1}$ as a test function in \eqref{fp:ex:iii:eq:b}, we obtain from $\negat{\psiL{n-1}} = 0$, $\negat{\QZL(\psi)}$ $= 0$, $\uflow$ being divergence-free in $\Omega$ and having zero trace on $\partial \Omega$ that
	\begin{align}
    \label{fp:ex:iii:nonneg}
    \begin{split}
	&\intOmSph \negat{\psi}^2 \dH\dx + \frac{\tau}{\DE} \intOmSph \left[ \mollEps^2 \abs{\nablax \negat{\psi}}^2 + \abs{\nablag \negat{\psi}}^2 \right] \dH\dx \\
	&\quad = \lambda \intOmSph \psiL{n-1} \negat{\psi} \dH\dx
	\leq 0.
    \end{split}
	\end{align}
	Hence, we have that $\psi \geq 0$ a.e. in $\OmSph$, that is $\psi \in \Zp{2}$.
	Moreover, taking $\theta \equiv 1$ in \eqref{fp:ex:iii:eq:b} yields
	\begin{align}
	\label{fp:ex:iii:mass_consv}
		\intOmSph \psi \dH\dx = \lambda \intOmSph \psiL{n-1} \dH\dx.
	\end{align}
	Hence, we infer from \eqref{fp:ex:iii:mass_consv}, \eqref{est:nabla_pot_Linfty}, and $0 < \lambda \leq 1$ that
	\begin{align}
	\label{fp:ex:iii:est_pot_infty}
		\norm{ \pot[\psi + \psiL{n-1}] }_{\W^{1, \infty}(\OmSph)} \leq C(\CmollEps, \potStr, d) \norm{\psiL{n-1}}_{\LpOS{1}}.
	\end{align}
	
	The goal is to prove the existence of a constant $C^\star > 0$, which is independent of $\lambda$ and $\psi$, such that
	\begin{align}
	\label{fp:ex:iii:goal}
	\norm{\psi}_{\LpOS{2}} \leq C^\star.
	\end{align}
	
	One common ansatz to prove \eqref{fp:ex:iii:goal} would be taking $\wflow := \uflow \in \Hdiv$ in \eqref{fp:ex:iii:eq:a} and $\theta := \psi \in \HkOS{1}$ in \eqref{fp:ex:iii:eq:b} as test functions to derive the desired bound. Then, combining \eqref{fp:iteration:iii:u} and \eqref{fp:iteration:iii:psi} yields
	\begin{align}
	\label{fp:ex:iii:fail}
	\norm{\psi}_{\LpOS{2}} \leq C(\tau, \regL) + C(\tau, \regL) \norm{\psi}_{\LpOS{2}} + C(\tau, \regL) \norm{\psi}_{\LpOS{2}}^2.
	\end{align}
	By choosing $\tau$ and $\regL$ in a clever way, we could absorb the second term on the right-hand side of \eqref{fp:ex:iii:fail} into the left-hand side. However, we could not control or absorb the last term on the right-hand side of \eqref{fp:ex:iii:fail}.
	This is the reason why we change the ansatz to mimicking the entropy estimate in Section \ref{sec:formal_entropy_estimate}. If we assume to control the $\regL$-regularized entropy of $\psi$ analogously to the formal estimate \eqref{formal:entropy_est}, we infer from \eqref{reg:FL_geq_ss} that
	\begin{align*}
	\frac{1}{4 \regL} \intOmSph \abs{\psi}^2 \dH\dx - C(\regL) \leq \intOmSph \entropL(\psi) \dH\dx \leq C,
	\end{align*}
	which then yields the desired bound \eqref{fp:ex:iii:goal}.
	Hence, our goal is to control the $\regL$-regularized entropy of $\psi$.
	
	As $\psi$ is nonnegative and there is no reason why $\psi$ should be strictly positive, we cannot apply $\dentropL(\cdot)$ to $\psi$ to define our test function as $\mu_{\psi, \regL} = \dentropL(\psi) + \pot[\psi]$. Hence, for $0 < \regDelta \ll 1$ let us define the $\regDelta$-regularized chemical potential by
	\begin{align}
	\label{fp:ex:iii:def:muPsidLl}
	\muPsidLl := \dentropL(\psi + \regDelta) + \frac{\lambda}{2} \pot\left[ \psi + \psiL{n-1} \right] \in \HkOS{1}.
	\end{align}
	In particular, the following representation of the diffusive terms motivates the multiplication of the interaction potential in \eqref{fp:ex:iii:def:muPsidLl} by $\lambda$.
	First, we rewrite the diffusive terms in \eqref{fp:ex:iii:eq:b}. It holds that, for all $\theta \in \HkOS{1}$,
	\begin{align}
	\label{fp:ex:iii:diff_x_rewritten}
	\begin{split}
	& \intOmSph \left[ \nablax \psi \cdot \nablax \theta + \frac{\lambda}{2} \QL(\psi) \nablax \pot[\psi + \psiL{n-1}] \cdot \nablax \theta \right] \dH\dx \\
	&\quad= \intOmSph \QL(\psi + \regDelta) \nablax \muPsidLl \cdot \nablax \theta \dH\dx \\
	&\quad\quad + \intOmSph \frac{\lambda}{2} \lrcb{ \QL(\psi) - \QL(\psi + \regDelta) } \nablax \pot[\psi + \psiL{n-1}] \cdot \nablax \theta \dH\dx.
	\end{split}
	\end{align}
	Similarly, we can rewrite the analogue term with respect to the surface gradient.
	Now, let us start deriving the entropy estimate.
	First, multiplying \eqref{fp:ex:iii:eq:a} by $\lambda$, taking $\wflow := \uflow \in \Hdiv$ as a test function in \eqref{fp:ex:iii:eq:a}, and noting the formula
	\begin{align}
	\label{eq:aab}
	a(a-b) = \frac{1}{2} \lrcb{ a^2 + (a-b)^2 - b^2 } \qquad \text{ for all } a,b \in \R,
	\end{align}
	it follows that
	\begin{align}
	\label{fp:ex:iii:u_tested_u}
	\begin{split}
	&\frac{\lambda \RE \DE}{2 \tau} \intOm \left[ \abs{\uflow}^2 + \abs{\uflow - \uL{n-1}}^2 - \abs{\uL{n-1}}^2 \right] \dx + \lambda \viscFrac \DE \intOm \abs{\nablax \uflow}^2 \dx \\
	&\quad + \frac{\lambda (1 - \viscFrac) \DE}{2} \intOmSph \QL(\psi) \lrcb{\nablax\uflow : \mom}^2 \dHm\dx \\
	&= - \frac{\lambda(1-\viscFrac)}{2} \intOmSph \psi \nablax \pot[\psi + \psiL{n-1}] \cdot \uflow \dH\dx \\
	&\quad - \frac{\lambda (1- \viscFrac)}{2} \intOmSph \QL(\psi) \lrcb{ (\IdMat - \mom)\nablax\uflow\mflow } \cdot \nablag \pot[\psi + \psiL{n-1}] \dHm\dx \\
	&\quad - \lambda(1 - \viscFrac) \intOmSph \psi (d\mom - \IdMat):\nablax\uflow \dHm\dx \\
	&\quad - \lambda\activity (1 - \viscFrac) \intOmSph \QL(\psi) \mom : \nablax \uflow \dHm\dx.
	\end{split}
	\end{align}
	
	Next, noting \eqref{fp:ex:iii:diff_x_rewritten}, the convexity of $\entropL$, and taking $\theta := \muPsidLl \in \HkOS{1}$ as a test function in \eqref{fp:ex:iii:eq:b}, we have that
	\begin{align*}
	&\frac{1}{\tau} \intOmSph \left[ \entropL(\psi + \regDelta) - \entropL(\lambda \psiL{n-1} + \regDelta) \right] \dH\dx \\
	&\quad\quad + \frac{\lambda}{2 \tau} \intOmSph \lrcb{\psi - \lambda \psiL{n-1}} \lrcb{ \pot[\psi] + \pot[\psiL{n-1}] } \dH\dx \\
	&\quad\quad + \frac{\lambda \mollEps^2}{\DE} \intOmSph \QL(\psi + \regDelta) \abs{\nablax \muPsidLl}^2 \dH\dx \\
	&\quad\quad + \frac{\lambda}{\DE} \intOmSph \QL(\psi + \regDelta) \abs{\nablag \muPsidLl}^2  \dH\dx \\
	&\quad\leq \frac{\lambda}{2} \intOmSph \psi \nablax \pot[\psi + \psiL{n-1}] \cdot \uflow \dH\dx \numberthis[fp:ex:iii:psi_tested_muPsidLl] \\
	&\quad\quad\quad + \lambda \intOmSph \QL(\psi) \lrcb{(\IdMat - \mom) \nablax \uflow \mflow} \cdot \nablag \muPsidLl \dHm\dx \\
	&\quad\quad\quad - \frac{\lambda \mollEps^2}{2 \DE} \intOmSph \lrcb{ \QL(\psi) - \QL(\psi + \regDelta) } \nablax \pot[\psi + \psiL{n-1}] \cdot \nablax \muPsidLl \dH\dx \\
	&\quad\quad\quad - \frac{\lambda}{2 \DE} \intOmSph \lrcb{\QL(\psi) - \QL(\psi + \regDelta)} \nablag \pot[\psi + \psiL{n-1}] \cdot \nablag \muPsidLl \dH\dx.
	\end{align*}
	
	Before we care about the error terms in \eqref{fp:ex:iii:psi_tested_muPsidLl}, we exploit the natural structure of \eqref{fp:ex:iii:u_tested_u} and \eqref{fp:ex:iii:psi_tested_muPsidLl} as good as possible to see which terms cancel out.
	We infer from \eqref{eq:chen_stress_formula} that the second term on the right-hand side of \eqref{fp:ex:iii:psi_tested_muPsidLl} can be rewritten as
	\begin{align*}
	&\lambda \intOmSph \QL(\psi) \lrcb{ (\IdMat - \mom) \nablax \uflow \mflow } \cdot \nablag \muPsidLl \dHm\dx\\
	&\quad= \frac{\lambda^2}{2} \intOmSph \QL(\psi) \lrcb{ (\IdMat - \mom) \nablax \uflow \mflow } \cdot \nablag \pot[\psi + \psiL{n-1}] \dHm\dx \\
	&\quad\quad + \lambda \intOmSph \QL(\psi) \lrcb{ (\IdMat - \mom) \nablax \uflow \mflow } \cdot \lrcb{ \frac{1}{\QL(\psi + \regDelta)} \nablag \psi } \dHm\dx \\
	&\quad= \frac{\lambda^2}{2} \intOmSph \QL(\psi) \lrcb{ (\IdMat - \mom) \nablax \uflow \mflow } \cdot \nablag \pot[\psi + \psiL{n-1}] \dHm\dx \\
	&\quad\quad + \lambda \intOmSph \psi (d\mom - \IdMat) : \nablax \uflow \dHm\dx \numberthis[fp:ex:iii:chen_stress_rewritten] \\
	&\quad\quad + \lambda \intOmSph \frac{\QL(\psi) - \QL(\psi + \regDelta)}{\QL(\psi + \regDelta)} \lrcb{ (\IdMat - \mom) \nablax \uflow \mflow } \cdot \nablag \psi \dHm\dx.
	\end{align*}
	We note that the second term on the right-hand side of \eqref{fp:ex:iii:chen_stress_rewritten} cancels out with the third term on the right-hand side of \eqref{fp:ex:iii:u_tested_u}.
	Furthermore, we deduce from \eqref{mollifier:fg}, $0 < \lambda \leq 1$, and the nonnegativity of the Maier-Saupe interaction kernel \eqref{def:interaction_kernel} and of $\psi$ and $\psiL{n-1}$ a.e. in $\OmSph$ that
	\begin{align}
	\label{fp:ex:iii:eq:pot_energy}
	\begin{split}
	&\frac{\lambda}{2\tau} \intOmSph \lrcb{\psi - \lambda \psiL{n-1}} \lrcb{ \pot[\psi] + \pot[\psiL{n-1}] } \dH\dx \\
	&\quad= \frac{\lambda}{2\tau} \intOmSph \left[ \psi\pot[\psi] - \lambda\psiL{n-1}\pot[\psiL{n-1}] \right] \dH\dx \\
	&\quad\quad + \frac{\lambda(1-\lambda)}{2 \tau} \intOmSph \psi \pot[\psiL{n-1}] \dH\dx \\
	&\quad\geq \frac{\lambda}{2\tau} \intOmSph \left[ \psi\pot[\psi] - \lambda\psiL{n-1}\pot[\psiL{n-1}] \right] \dH\dx.
	\end{split}
	\end{align}
	
	Now, we obtain from multiplying \eqref{fp:ex:iii:psi_tested_muPsidLl} by $(1 - \viscFrac)$ and from combining \eqref{fp:ex:iii:u_tested_u} and \eqref{fp:ex:iii:psi_tested_muPsidLl} - \eqref{fp:ex:iii:eq:pot_energy} that
	\begin{align*}
	&\frac{\lambda\RE\DE}{2 \tau} \intOm \left[ \abs{\uflow}^2 + \abs{\uflow - \uL{n-1}}^2 - \abs{\uL{n-1}}^2 \right] \dx + \lambda \viscFrac \DE \intOm \abs{\nablax \uflow}^2 \dx \\
	&\quad + \frac{\lambda (1 - \viscFrac) \DE}{2} \intOmSph \QL(\psi) \lrcb{\nablax\uflow : \mom}^2 \dHm\dx \\
	&\quad + \frac{1 - \viscFrac}{\tau} \intOmSph \left[ \entropL(\psi + \regDelta) - \entropL(\lambda \psiL{n-1} + \regDelta) \right] \dH\dx \\
	&\quad + \frac{\lambda (1 - \viscFrac)}{2 \tau} \intOmSph \left[ \psi \pot[\psi] - \lambda \psiL{n-1} \pot[\psiL{n-1}] \right] \dH\dx \\
	&\quad + \frac{\lambda (1 - \viscFrac) \mollEps^2}{\DE} \intOmSph \QL(\psi + \regDelta) \abs{\nablax \muPsidLl}^2 \dH\dx \\
	&\quad + \frac{\lambda (1 - \viscFrac)}{\DE} \intOmSph \QL(\psi + \regDelta) \abs{\nablag \muPsidLl}^2 \dH\dx \\	
	&\leq - \lambda \alpha (1 - \viscFrac) \intOmSph \QL(\psi) \mom : \nablax \uflow \dHm\dx \numberthis[fp:ex:iii:u_psi_combined] \\
	&\quad + \frac{\lambda(\lambda-1)(1-\viscFrac)}{2} \intOmSph \QL(\psi) \lrcb{ (\IdMat - \mom) \nablax \uflow \mflow } \cdot  \\
	&\qquad\qquad\qquad\qquad\qquad\qquad\qquad\qquad\qquad \cdot \nablag \pot[\psi + \psiL{n-1}] \dHm\dx \\
	&\quad + \lambda(1 - \viscFrac) \intOmSph \frac{\QL(\psi) - \QL(\psi + \regDelta)}{\QL(\psi + \regDelta)} \lrcb{ (\IdMat - \mom) \nablax \uflow \mflow } \cdot \\
	&\qquad\qquad\qquad\qquad\qquad\qquad\qquad\qquad\qquad \cdot \nablag \psi \dHm\dx \\
	&\quad - \frac{\lambda \mollEps^2 (1 - \viscFrac)}{2 \DE} \intOmSph \lrcb{ \QL(\psi) - \QL(\psi + \regDelta) } \nablax \pot[\psi + \psiL{n-1}] \cdot  \\
	&\qquad\qquad\qquad\qquad\qquad\qquad\qquad\qquad\qquad \cdot \nablax \muPsidLl \dH\dx \\
	&\quad - \frac{\lambda (1 - \viscFrac)}{2 \DE} \intOmSph \lrcb{\QL(\psi) - \QL(\psi + \regDelta)} \nablag \pot[\psi + \psiL{n-1}] \cdot  \\
	&\qquad\qquad\qquad\qquad\qquad\qquad\qquad\qquad\qquad \cdot \nablag \muPsidLl \dH\dx \\
	&=: I + II + III + IV + V.
	\end{align*}
	We note that the active term $I$ appears in the same form as in the formal estimate \eqref{formal:summarize_1}. Moreover, it is worth mentioning that in the case of the usual entropy estimate, that is $\lambda = 1$, we would have that $II = 0$.
	The error terms $III-V$ are caused by the approximate scheme and the $\regDelta$-regularized test function $\muPsidLl$.\\
	Let us now continue estimating the terms $I - V$. We infer from $0 < \lambda \leq 1$, \eqref{reg:QL_leq}, \eqref{fp:ex:iii:mass_consv}, and Young's inequality that
	\begin{align*}
	\abs{I}
	&\leq \frac{\lambda (1 - \viscFrac) \DE}{4} \intOmSph \QL(\psi) \lrcb{\nablax\uflow : \mom}^2 \dHm\dx \\
	&\quad + \frac{(1-\viscFrac) \activity^2}{\DE} \norm{\psiL{n-1}}_{\LpOS{1}}. \numberthis[fp:ex:iii:I]
	\end{align*}
	The first term on the right-hand side of \eqref{fp:ex:iii:I} can be absorbed into the left-hand side of \eqref{fp:ex:iii:u_psi_combined}.
	Furthermore, we deduce from $\abs{\IdMat - \mom} \leq 2d$, $\abs{\lambda(\lambda-1)} \leq \lambda$, \eqref{reg:QL_leq}, \eqref{fp:ex:iii:mass_consv}, \eqref{fp:ex:iii:est_pot_infty}, Jensen's inequality, and Young's inequality that
	\begin{align}
	\label{fp:ex:iii:II}
	\begin{split}
	\abs{II}
	&\leq \frac{\lambda (1-\viscFrac) 2d }{2} \intOmSph \abs{\QL(\psi)} \abs{\nablax \uflow} \abs{\nablag \pot[\psi + \psiL{n-1}]} \dH\dx \\
	&\leq \frac{\lambda \gamma \DE}{2} \intOm \abs{\nablax \uflow}^2 \dx  + C_1 \regL \norm{\psiL{n-1}}_{\LpOS{1}}^3,
	\end{split}
	\end{align}
	where $C_1 = C_1(\viscFrac, \viscFrac^{-1}, \DE^{-1}, d, \OmSph, \CmollEps, \potStr) > 0$ is independent of $\regDelta, \tau$, $\lambda$, $\regL$, and $\psi$.
	Again, we can absorb the first term on the right-hand side of \eqref{fp:ex:iii:II} into the left-hand side of \eqref{fp:ex:iii:u_psi_combined}.
	Moreover, on noting the Lipschitz continuity of $\QL$, the nonnegativity of $\psi$, \eqref{reg:QL_leq}, \eqref{fp:ex:iii:mass_consv}, \eqref{fp:ex:iii:est_pot_infty}, $0 < \lambda \leq 1$, and on adapting a computation of Chen and Liu (cf. \cite[(3.32)]{chen_liu_2013_entropy_solution}), we infer from Jensen's inequality and Young's inequality that
	\begin{align*}
	\abs{III}
	&\leq \lambda (1 - \viscFrac) 2d \sqrt{\regDelta} \intOmSph \abs{\nablax\uflow} \abs{\nablag \psi} \frac{1}{\sqrt{\QL(\psi + \regDelta)}} \dH\dx \\
	&\leq \frac{\lambda (1 - \viscFrac)}{8 \DE} \intOmSph \QL(\psi + \regDelta) \abs{\nablag \dentropL(\psi + \regDelta)}^2 \dH\dx \\
	&\quad + C_2(\DE, \viscFrac, d) \regDelta \intOm \abs{\nablax \uflow}^2 \dx \\
	&\leq \frac{\lambda (1 - \viscFrac)}{4 \DE} \intOmSph \QL(\psi + \regDelta) \abs{\nablag \muPsidLl}^2 \dH\dx \numberthis[fp:ex:iii:III] \\
	&\quad + C_3\lrcb{\DE^{-1}, \viscFrac, \CmollEps, \potStr, d, \Omega} \norm{\psiL{n-1}}_{\LpOS{1}}^2 \lrcb{\norm{\psiL{n-1}}_{\LpOS{1}} + 1} \\
	&\quad + C_2(\DE, \viscFrac, d) \regDelta \intOm \abs{\nablax \uflow}^2 \dx,
	\end{align*}
	where $C_2 = C_2(\DE, \viscFrac, d) > 0$ and $C_3 =C_3\lrcb{\DE^{-1}, \viscFrac, \CmollEps, \potStr, d, \Omega} > 0$ are independent of $\regDelta, \tau$, $\lambda$, $\regL$, and $\psi$.
	We note that the first term on the right-hand side of \eqref{fp:ex:iii:III} can be absorbed into the left-hand side of \eqref{fp:ex:iii:u_psi_combined}. Moreover, the last term on the right-hand side of \eqref{fp:ex:iii:III} can be absorbed as well if we choose $\regDelta > 0$ sufficiently small. The exact range of $\regDelta$ will be specified later.
	Next, we infer from $\QL(s) \leq \QL(s + \regDelta)$ for all $s \geq 0$, the Lipschitz continuity of $\QL$, \eqref{reg:QL_leq}, \eqref{fp:ex:iii:mass_consv}, \eqref{fp:ex:iii:est_pot_infty}, $0 < \lambda \leq 1$, Jensen's inequality, and Young's inequality that
	\begin{align*}
	\abs{IV}
	&\leq \frac{\lambda \mollEps^2 (1 - \viscFrac)}{2 \DE} \intOmSph \QL(\psi + \regDelta) \abs{\nablax \muPsidLl}^2 \dH\dx \numberthis[fp:ex:iii:IV] \\
	&\quad + C_4 \lrcb{\mollEps, \DE^{-1}, \viscFrac, \Omega, \CmollEps, \potStr, d} \regDelta \norm{\psiL{n-1}}_{\LpOS{1}}^2
	\end{align*}
	and
	\begin{align*}
	\abs{V}
	&\leq \frac{\lambda (1 - \viscFrac)}{4 \DE} \intOmSph \QL(\psi + \regDelta) \abs{\nablag \muPsidLl}^2 \dH\dx \numberthis[fp:ex:iii:V] \\
	&\quad + C_5\lrcb{\DE^{-1}, \viscFrac, \Omega, \CmollEps, \potStr, d} \regDelta \norm{\psiL{n-1}}_{\LpOS{1}}^2,
	\end{align*}
	where $C_4 = C_4 \lrcb{\mollEps, \DE^{-1}, \viscFrac, \Omega, \CmollEps, \potStr, d} > 0$ and $C_5 = C_5\lrcb{\DE^{-1}, \viscFrac, \Omega, \CmollEps, \potStr, d} > 0$ are independent of $\regDelta, \tau$, $\lambda$, $\regL$, and $\psi$.
	Now, we obtain from combining \eqref{fp:ex:iii:u_psi_combined} - \eqref{fp:ex:iii:V} that
	\begin{align*}
	&\frac{\lambda \RE \DE}{2 \tau} \intOm \left[ \abs{\uflow}^2 + \abs{\uflow - \uL{n-1}}^2 - \abs{\uL{n-1}}^2 \right] \dx + \lrcb{\frac{\lambda\viscFrac\DE}{2} - C_2 \regDelta} \intOm \abs{\nablax \uflow}^2 \dx \\
	&\quad\quad + \frac{\lambda (1 - \viscFrac) \DE}{4} \intOmSph \QL(\psi) \lrcb{\nablax\uflow : \mom}^2 \dHm\dx \\
	&\quad\quad + \frac{1 - \viscFrac}{\tau} \intOmSph \left[ \entropL(\psi + \regDelta) - \entropL(\lambda\psiL{n-1} + \regDelta) \right] \dH\dx \\
	&\quad\quad + \frac{\lambda (1 - \viscFrac)}{2 \tau} \intOmSph \left[ \psi\pot[\psi] - \lambda\psiL{n-1}\pot[\psiL{n-1}] \right] \dH\dx \\
	&\quad\quad + \frac{\lambda (1 - \viscFrac) \mollEps^2}{2 \DE} \intOmSph \QL(\psi + \regDelta) \abs{\nablax \muPsidLl}^2 \dH\dx \\
	&\quad\quad + \frac{\lambda (1 - \viscFrac)}{2 \DE} \intOmSph \QL(\psi + \regDelta) \abs{\nablag \muPsidLl}^2 \dH\dx \\
	&\quad\leq \frac{(1-\viscFrac) \activity^2}{\DE} \norm{\psiL{n-1}}_{\LpOS{1}} +  C_1 \regL \norm{\psiL{n-1}}_{\LpOS{1}}^3 \numberthis[fp:ex:iii:u_psi_I_to_V] \\
	&\quad\quad + C_3 \lrcb{\norm{\psiL{n-1}}_{\LpOS{1}} + 1}\norm{\psiL{n-1}}_{\LpOS{1}}^2 \\
	&\quad\quad + \lrcb{C_4 + C_5 }\regDelta \norm{\psiL{n-1}}_{\LpOS{1}}^2.
	\end{align*}

    Now, we deduce from \eqref{fp:ex:iii:u_psi_I_to_V}, $0 < \lambda \leq 1$, \eqref{reg:FL_geq_ss}, the nonnegativity of $\psi$, omitting some nonnegative terms, and the Dominated Convergence Theorem that
	\begin{align}
	\label{fp:ex:iii:bound_psi_psi_1}
	\begin{split}
	\frac{1}{4 \regL} \intOmSph \abs{\psi}^2 \dH\dx 
	 \leq C,
	\end{split}
	\end{align}
	where $C = C\lrcb{C_1, C_3, \potStr, d, \Omega,  \RE, \DE, \DE^{-1}, \activity, \viscFrac, \tau, \regL, \uL{n-1}, \psiL{n-1}} > 0$ is independent of $\lambda$ and $\psi$.
	Hence, it finally follows that \eqref{fp:ex:iii:goal} holds. The proof of Lemma \ref{lmm:existence_fp} is thereby complete.
\end{proof}

\begin{crllr}
	\label{crllr:nonnegativity}
	Suppose that the assumptions \eqref{assumptions:domain_init_data} hold.
	Then, there exist solutions $\{ ( \uL{k}, \psiL{k} ) \}_{k=1}^N$ $\subset$ $\Hdiv \times \HkOS{1}$ to \eqref{def:PLtau}. In particular, it holds that $\psiL{k} \in \Zp{2}$ for $k \in \{1, \dots, N\}$. Hence, $\{ ( \uL{k}, \psiL{k} ) \}_{k=1}^N$ $\subset$ $\Hdiv \times (\HkOS{1} \cap \Zp{2})$ is a solution to $\PLtau$.
\end{crllr}

\begin{proof}
	As a fixed-point of $\Xi$ is equivalent to a solution to \eqref{def:PLtau}, applying Lemma \ref{lmm:existence_fp} iteratively yields the existence of solutions $\{ ( \uL{k}, \psiL{k} ) \}_{k=1}^N$ $\subset$ $\Hdiv \times \H^1(\Omega \times \Sph)$ to \eqref{def:PLtau}.
	It therefore remains to show that $\{\psiL{k}\}_{k=1}^N \subset \Zp{2}$. 
	As $\psiL{0} \in \Zp{2}$ (see \eqref{init:def:psi_L_0}), let $k \in \{1, \dots, N\}$ and $\psiL{k-1} \in \Zp{2}$.
    Applying the same arguments as in \eqref{fp:ex:iii:nonneg} yields $\psiL{k} \in \Zp{2}$.
\end{proof}

\subsection{A uniform entropy estimate}
\label{sec:uniform_entropy_estimate}
This section is devoted to the derivation of a uniform entropy estimate with respect to $N$ and $\regL$, which mimics the formal entropy estimate \eqref{formal:entropy_est}.
First, we infer from taking $\psi \equiv 1$ as a test function in \eqref{eq:PLtau:psi} that the mass of $\psiL{n}$ is conserved, that is
\begin{align}
\label{eq:mass_consv_psiLn}
\intOmSph \psiL{n} \dH\dx = \intOmSph \psiL{n-1} \dH\dx = \intOmSph \psiL{0} \dH\dx
\end{align}
for all $n \in \{1, \dots, N\}$.
Moreover, for future references, we deduce from \eqref{est:nabla_pot_Linfty}, \eqref{eq:mass_consv_psiLn}, and \eqref{init:regularity_psi_L_0} that, for all $k \in \{1, \dots, N\}$,
\begin{align}
\label{eq:nabla_potLk_bound}
\begin{split}
\norm{\potL{k}}_{\W^{1, \infty}(\OmSph)}
\leq C(\CmollEps, \potStr, d) \norm{\psi_0}_{\LpOS{1}}
< \infty,
\end{split}
\end{align}
where $C = C(\CmollEps, \potStr, d) > 0$ is independent of $\regDelta, \tau$, and $\regL$. 
The following lemma holds:

\begin{lmm}[Discrete-in-time entropy estimate]
	\label{lmm:time_discrete_entropy}
	Suppose that the assumptions \eqref{assumptions:domain_init_data} hold.
	Then, solutions $\{ ( \uL{k}, \psiL{k} ) \}_{k=1}^N$ $\subset$ $\Hdiv \times (\HkOS{1} \cap \Zp{2})$ to \eqref{def:PLtau} satisfy
	\begin{align*}
	& \frac{\RE \DE}{2} \intOm \abs{ \uL{n} }^2 \dx + \frac{\RE\DE}{2} \sum_{k=1}^{n} \intOm \abs{\uL{k} - \uL{k-1}}^2 \dx + \frac{\viscFrac \DE}{2} \sum_{k=1}^{n} \tau \intOm \abs{\nablax \uL{k}}^2 \dx \\
	& \quad + \frac{(1 - \viscFrac) \DE}{4} \sum_{k=1}^{n} \tau \intOmSph \QL(\psiL{k}) \lrcb{ \nablax \uL{k} : \mom }^2 \dHm \dx \\
	& \quad + (1 - \viscFrac) \intOmSph \entrop(\psiL{n}) \dH\dx + \frac{1 - \viscFrac}{2} \intOmSph \psiL{n} \pot[\psiL{n}] \dH \dx \\
	& \quad + \frac{1 - \viscFrac}{2 \regL} \sum_{k=1}^n \intOmSph \abs{ \psiL{k} - \psiL{k-1} }^2 \dH\dx \\
	& \quad + \frac{1 - \viscFrac}{\DE} \sum_{k=1}^{n} \tau \intOmSph \Big[ \mollEps^2 | \nablax \sqrt{\psiL{k}} |^2 + | \nablag \sqrt{\psiL{k}} |^2 \Big] \dH\dx \\
	&\leq \frac{\RE \DE}{2} \intOm \abs{\uflow_0}^2 \dx + (1 - \viscFrac) \intOmSph \entrop(\psi_0) \dH\dx \numberthis[est:discrete_entropy] \\
	&\quad + \frac{1 - \viscFrac}{2} \intOmSph \psi_0 \pot[\psi_0] \dH\dx + \frac{\activity^2 (1 - \viscFrac) T}{\DE} \norm{\psi_0}_{\LpOS{1}} + C_{*} \\
	&\leq C\lrcb{ \RE, \DE, \DE^{-1}, \viscFrac, \activity, T, \CmollEps, \potStr, d, \Omega, \mollEps, \uflow_0, \psi_0 }
	< \infty \qquad \text{ for all } n \in \{ 1, \dots, N \},
	\end{align*}
	where $C_{*} > 0$ depends solely on $T, \psi_0, \DE^{-1}, \viscFrac, \CmollEps, \potStr, d, \Omega$, and $\mollEps$. 
\end{lmm}

\begin{proof}[Proof of Lemma \ref{lmm:time_discrete_entropy}]
	Let a regularization parameter $0 < \regDelta \ll 1$ be given. Then, we rewrite the diffusive terms in \eqref{eq:PLtau:psi} for all $\theta \in \HkOS{1}$ analogously to \eqref{fp:ex:iii:diff_x_rewritten} as follows:
	\begin{align}
    \label{discr_entr:diff_rewritten}
    \begin{split}
	& \frac{\mollEps^2}{\DE} \intOmSph \lrcb{\nablax \psiL{n} + \QL(\psiL{n}) \nablax \potL{n}} \cdot \nablax \theta \dH\dx \\
	&\quad = \frac{\mollEps^2}{\DE} \intOmSph \QL(\psiL{n} + \regDelta) \nablax \muPsidL{n} \cdot \nablax \theta \dH\dx \\
	&\quad \quad + \frac{\mollEps^2}{\DE} \intOmSph \lrcb{\QL(\psiL{n}) - \QL(\psiL{n} + \regDelta)} \nablax \potL{n} \cdot \nablax \theta \dH\dx,
    \end{split}
	\end{align}
	where we introduced the $\regDelta$-regularized chemical potential
	\begin{align*}
	\muPsidL{n} := \dentropL(\psiL{n} + \regDelta) + \potL{n} \quad \in \HkOS{1}.
	\end{align*}
    
    Moreover, we recall the following estimate, which has been derived by Barrett and S\"uli by Taylor's series expansion (cf. \cite[p. 1238]{barrett_sueli_2011_polymers_I}):
	\begin{align}
    \label{discr_entr:taylor_est}
    \begin{split}
	&\intOmSph \frac{\psiL{n} - \psiL{n-1}}{\tau} \dentropL(\psiL{n} + \regDelta) \dH\dx \\
	&\quad\geq \frac{1}{\tau} \intOmSph \left[ \entropL(\psiL{n} + \regDelta) - \entropL(\psiL{n-1} + \regDelta) \right] \dH\dx \\
	&\quad\quad + \frac{1}{2 \tau \regL}\intOmSph \abs{ \psiL{n} - \psiL{n-1} } \dH\dx.
	\end{split}
    \end{align}

    Now, on noting \eqref{discr_entr:diff_rewritten} and \eqref{discr_entr:taylor_est}, we infer from taking $\wflow := \uL{n} \in \Hdiv$ and $\theta := \muPsidL{n} \in \H^1(\Omega \times \Sph)$ as test functions in \eqref{eq:PLtau:u} and \eqref{eq:PLtau:psi}, respectively, and from the analogue estimate \eqref{fp:ex:iii:u_psi_combined} with $\lambda = 1$ that the following estimate holds:
	\begin{align*}
	&\frac{\RE\DE}{2 \tau} \intOm \left[ \abs{\uL{n}}^2 + \abs{\uL{n} - \uL{n-1}}^2 - \abs{\uL{n-1}}^2 \right] \dx + \viscFrac\DE \intOm \abs{\nablax \uL{n}}^2 \dx \\
	&\qquad + \frac{(1-\viscFrac) \DE}{2} \intOmSph \QL(\psiL{n}) \lrcb{\nablax \uL{n} : \mom}^2 \dHm\dx \\
	&\qquad + \frac{1-\viscFrac}{\tau} \intOmSph \left[ \entropL(\psiL{n} + \regDelta) - \entropL(\psiL{n-1} + \regDelta) \right] \dH\dx \\
	&\qquad + \frac{1-\viscFrac}{2 \tau} \intOmSph \left[ \psiL{n} \pot[\psiL{n}] - \psiL{n-1} \pot[\psiL{n-1}] \right]\dH\dx \\
	&\qquad + \frac{1-\viscFrac}{2 \tau \regL} \intOmSph \abs{ \psiL{n} - \psiL{n-1} }^2 \dH\dx \\
	&\qquad + \frac{(1-\viscFrac) \mollEps^2}{\DE} \intOmSph \QL(\psiL{n} + \regDelta) \abs{\nablax \muPsidL{n}}^2 \dH\dx \\
	&\qquad + \frac{1-\viscFrac}{\DE} \intOmSph \QL(\psiL{n} + \regDelta) \abs{\nablag \muPsidL{n}}^2 \dH\dx \\
	&\quad\leq - \activity (1 - \viscFrac) \intOmSph \QL(\psiL{n}) \mom : \nablax \uL{n} \dHm\dx \numberthis[discr_entr:u_psi_combined] \\
	&\quad\quad + (1-\viscFrac) \intOmSph \frac{\QL(\psiL{n}) - \QL(\psiL{n} + \regDelta)}{\QL(\psiL{n} + \regDelta)} \lrcb{ (\IdMat - \mom) \nablax \uL{n} \mflow } \cdot  \\
	&\qquad\qquad\qquad\qquad\qquad\qquad\qquad\qquad\qquad\qquad\qquad \cdot \nablag \psiL{n} \dHm\dx \\
	&\quad\quad - \frac{(1-\viscFrac)\mollEps^2}{\DE} \intOmSph \lrcb{ \QL(\psiL{n}) - \QL(\psiL{n} + \regDelta) } \nablax \potL{n} \cdot \nablax \muPsidL{n} \dH\dx \\
	&\quad\quad - \frac{1-\viscFrac}{\DE} \intOmSph \lrcb{ \QL(\psiL{n}) - \QL(\psiL{n} + \regDelta) } \nablag \potL{n} \cdot \nablag \muPsidL{n} \dH\dx \\
	&\quad=: I + II + III + IV.
	\end{align*}
	
	We infer from \eqref{reg:QL_leq}, the nonnegativity of $\psiL{n}$, and \eqref{eq:mass_consv_psiLn} that
	\begin{align}
	\label{discr_entr:I}
	\begin{split}
	\abs{I}
	&\leq \frac{(1-\viscFrac)\DE}{4} \intOmSph \QL(\psiL{n}) \lrcb{\nablax\uL{n} : \mom}^2 \dHm\dx \\
	&\quad + \frac{\activity^2 (1 - \viscFrac)}{\DE} \intOmSph \psiL{0} \dH\dx.
	\end{split}
	\end{align}
	The first term on the right-hand side of \eqref{discr_entr:I} can be absorbed into the left-hand side of \eqref{discr_entr:u_psi_combined}.
	Next, similarly to \eqref{fp:ex:iii:III}, we deduce from the Lipschitz continuity of $\QL$, the nonnegativity of $\psiL{n}$, \eqref{eq:nabla_potLk_bound}, \eqref{reg:QL_leq}, \eqref{eq:mass_consv_psiLn}, $0 < \regDelta < 1$, and Young's inequality that
	\begin{align}
    \label{discr_entr:II}
    \begin{split}
	\abs{II}
	&\leq \frac{1-\viscFrac}{4 \DE} \intOmSph \QL(\psiL{n} + \regDelta) \abs{\nablag \muPsidL{n}}^2 \dH\dx  \\
	&\quad + C_2\lrcb{\DE^{-1}, \viscFrac, \CmollEps, \potStr, d, \Omega} \norm{\psiL{0}}_{\LpOS{1}}^2 \lrcb{ 1 + \norm{\psiL{0}}_{\LpOS{1}} } \\
	&\quad + C_1(\DE, \viscFrac, d) \regDelta \intOm \abs{\nablax \uL{n}}^2 \dx,
    \end{split}
	\end{align}
	where $C_1 = C_1(\DE, \viscFrac, d) > 0$ and $C_2 = C_2\lrcb{\DE^{-1}, \viscFrac, \CmollEps, \potStr, d, \Omega} > 0$ are independent of $\regDelta$, $\tau$, and $\regL$.
	Furthermore, on noting $\QL(s) \leq \QL(s + \regDelta)$ for all $s \geq 0$, the Lipschitz continuity of $\QL$, the nonnegativity of $\psiL{n}$, \eqref{eq:nabla_potLk_bound}, and Young's inequality, we have that
	\begin{align}
	\label{discr_entr:III}
	\begin{split}
	\abs{III}
	&\leq \frac{\mollEps^2 (1-\viscFrac)}{2 \DE} \intOmSph \QL(\psiL{n} + \regDelta) \abs{\nablax \muPsidL{n}}^2 \dH\dx \\
	&\quad + C_3\lrcb{ \DE^{-1}, \mollEps, \viscFrac, \Omega, \CmollEps, \potStr, d } \regDelta \norm{\psiL{0}}_{\LpOS{1}}^2,
	\end{split}
	\end{align}
	where $C_3 = C_3\lrcb{ \DE^{-1}, \mollEps, \viscFrac, \Omega, \CmollEps, \potStr, d } > 0$ is independent of $\regDelta$, $\tau$, and $\regL$.
	Analogously to \eqref{discr_entr:III}, it follows that
	\begin{align*}
	\abs{IV}
	&\leq \frac{1-\viscFrac}{4 \DE} \intOmSph \QL(\psiL{n} + \regDelta) \abs{\nablag \muPsidL{n}}^2 \dH\dx \\
	&\quad + C_4\lrcb{\DE^{-1}, \viscFrac, \Omega, \CmollEps, \potStr, d} \regDelta \norm{\psiL{0}}_{\LpOS{1}}^2, \numberthis[discr_entr:IV]
	\end{align*}
	where $C_4 = C_4\lrcb{\DE^{-1}, \viscFrac, \Omega, \CmollEps, \potStr, d} > 0$ is independent of $\regDelta$, $\tau$, and $\regL$.
	Now, on combining \eqref{discr_entr:u_psi_combined}-\eqref{discr_entr:IV}, restricting the range of $\regDelta$ to
	\begin{align*}
	0 < \regDelta < \frac{\viscFrac \DE}{2 C_1},
	\end{align*}
	and applying a discrete integration in time, we have that
	\begin{align*}
	&\frac{\RE \DE}{2} \intOm \abs{\uL{n}}^2 \dx + \frac{\RE \DE}{2} \sum_{k=1}^{n} \intOm \abs{\uL{k} - \uL{k-1}}^2 \dx + \frac{\viscFrac \DE}{2} \sum_{k=1}^{n} \tau \intOm \abs{\nablax \uL{k}}^2 \dx \\
	& \qquad + \frac{(1-\viscFrac) \DE}{4} \sum_{k=1}^{n} \tau \intOmSph \QL(\psiL{k}) \lrcb{ \nablax \uL{k} : \mom }^2 \dHm\dx \\
	& \qquad + (1-\viscFrac) \intOmSph \entropL(\psiL{n} + \regDelta) \dH\dx + \frac{1-\viscFrac}{2} \intOmSph \psiL{n} \pot[\psiL{n}] \dH\dx \\
	&\qquad + \frac{1-\viscFrac}{2 \regL} \sum_{k=1}^n \intOmSph \abs{ \psiL{k} - \psiL{k-1} }^2 \dH\dx \\
	& \qquad + \frac{(1-\viscFrac) \mollEps^2}{2 \DE} \sum_{k=1}^{n} \tau \intOmSph \QL(\psiL{k} + \regDelta) \abs{ \nablax \muPsidL{k} }^2 \dH\dx \\
	& \qquad + \frac{1-\viscFrac}{2 \DE} \sum_{k=1}^{n} \tau \intOmSph \QL(\psiL{k} + \regDelta) \abs{ \nablag \muPsidL{k} }^2 \dH\dx \\
	&\quad\leq \frac{\RE \DE}{2} \intOm \abs{\uL{0}}^2 \dx + (1-\viscFrac) \intOmSph \entropL(\psiL{0} + \regDelta) \dH\dx \numberthis[discr_entr:u_psi_I_IV_combined] \\
	&\quad\quad + \frac{1-\viscFrac}{2} \intOmSph \psiL{0} \pot[\psiL{0}] \dH\dx + \frac{\activity^2 (1 - \viscFrac) T}{\DE} \norm{\psiL{0}}_{\LpOS{1}} \\
	& \quad\quad + C_2 T \norm{\psiL{0}}_{\LpOS{1}}^2 \lrcb{ 1 + \norm{\psiL{0}}_{\LpOS{1}} } \\
	& \quad\quad + (C_3 + C_4) T \regDelta \norm{\psiL{0}}_{\LpOS{1}}^2 \qquad \qquad \qquad  \text{ for all } n \in \{ 1, \dots, N \}.
	\end{align*}
	
    Next, we infer from \eqref{eq:nabla_potLk_bound}, \eqref{reg:FLQL_leq_FL}, \eqref{eq:mass_consv_psiLn}, \eqref{init:regularity_psi_L_0}, and applying Lemma of Fatou (cf. \cite[p. 24]{chen_liu_2013_entropy_solution}, \cite[p. 1244f]{barrett_sueli_2011_polymers_I}) that
    \begin{align}
    \label{discr_entr:limit:diffusive_x}
    \begin{split}
    & \frac{(1-\viscFrac)\mollEps^2}{2 \DE} \lim_{\regDelta \searrow 0} \intOmSph \QL(\psiL{k} + \regDelta) \abs{ \nablax \muPsidL{k} }^2 \dH\dx \\
    &\quad\geq \frac{(1-\viscFrac)\mollEps^2}{4 \DE} \lim_{\regDelta \searrow 0} \intOmSph \frac{\abs{\nablax \psiL{k}}^2}{\psiL{k} + \regDelta} \dH\dx - C(\CmollEps, \potStr, d, \Omega, \viscFrac, \mollEps, \DE^{-1}, \psi_0) \\
    &\quad\geq \frac{(1-\viscFrac)\mollEps^2}{\DE} \intOmSph \abs{\nablax \sqrt{\psiL{k}}}^2 \dH\dx - C(\CmollEps, \potStr, d, \Omega, \viscFrac, \mollEps, \DE^{-1}, \psi_0).
    \end{split}
    \end{align}
    We apply similar computations for the term $\QL(\psiL{k} + \regDelta) \abs{\nablag \muPsidL{k}}^2$ (cf. \cite[p. 24]{chen_liu_2013_entropy_solution}). Finally, on noting \eqref{discr_entr:limit:diffusive_x}, the regularities of the initial data \eqref{init:regularity_u_L_0}, \eqref{init:regularity_psi_L_0}, \eqref{init:psi_0_pot_0}, and \eqref{reg:FLQL_leq_FL}, applying Lemma of Fatou and the Dominated Convergence Theorem in \eqref{discr_entr:u_psi_I_IV_combined} yields the entropy estimate \eqref{est:discrete_entropy} when passing to the limit as $\regDelta \searrow 0$.
\end{proof}

\subsection{Regularity of the polymer number density}
\label{sec:particle_density}
The discrete-in-time entropy estimate \eqref{est:discrete_entropy} is a good starting point to establish regularity results for $\uL{n}$ and $\psiL{n}$. However, the bounds in \eqref{est:discrete_entropy} are too weak to extract strongly converging subsequences. Unfortunately, testing \eqref{eq:PLtau:psi} with $\theta := \psiL{n}$ is fraud with difficulties to improve the regularity of $\psiL{n}$.
Instead, we shall establish bounds on the discrete-in-time polymer number density
\begin{align}
\label{def:omegaL}
\omegaL{n}(\xflow) := \intSph \psiL{n}(\xflow, \mflow) \dHm \quad \text{ for } \xflow \in \Omega,\ n \in \{0, \dots, N\}.
\end{align}

\begin{lmm}
	\label{lmm:regularity_omegaL}
	Suppose that the assumptions \eqref{assumptions:domain_init_data} hold.
	Moreover, let the time increment satisfy 
	\begin{align*}
	0 < \tau < \frac{1}{4 \widetilde{C} \norm{\psi_0}_{\LpOS{1}}^2},
	\end{align*}
	where $\widetilde{C} > 0$ depends only on $\DE^{-1}, \mollEps, \CmollEps, \potStr$, and $d$. Then, the following estimate holds for $n \in \{1, \dots, N\}$
	\begin{align}
	\label{eq:regularity_omegaL}
	\begin{split}
	&\frac{1}{4} \intOm \abs{\omegaL{n}}^2 \dx + \frac{1}{2} \sum_{k=1}^{n} \intOm \abs{\omegaL{k} - \omegaL{k-1}}^2 \dx + \frac{\mollEps^2}{2 \DE} \sum_{k=1}^{n} \tau \intOm \abs{ \nablax \omegaL{k} }^2 \dx \\
	&\quad\leq \lrcb{ 1 + \widetilde{C} \norm{\psi_0}_{\LpOS{1}}^2 T \exp\lrcb{ \widetilde{C} \norm{\psi_0}_{\LpOS{1}}^2 T } } \norm{\psi_0}_{\L^2(\Omega; \L^1(\Sph))}^2 \\
	&\quad< \infty.
	\end{split}
	\end{align}
\end{lmm}

\begin{proof}
	We deduce from choosing $\theta := \omegaL{n} \in \H^1(\Omega)$ in \eqref{eq:PLtau:psi}, \eqref{eq:nabla_potLk_bound}, \eqref{reg:QL_leq}, \eqref{eq:mass_consv_psiLn}, \eqref{def:omegaL},
	\begin{align*}
	\intOm \omegaL{n} \nablax \omegaL{n} \cdot \uL{n} \dx = \intOm \frac{1}{2} \nablax (\omegaL{n})^2 \cdot \uL{n} \dx = 0,
	\end{align*}
	and Young's inequality that
	\begin{align*}
	&\frac{1}{2 \tau} \intOm \left[ \abs{\omegaL{n}}^2 + \abs{\omegaL{n} - \omegaL{n-1}}^2  - \abs{\omegaL{n-1}}^2 \right] \dx + \frac{\mollEps^2}{\DE} \intOm \abs{ \nablax \omegaL{n} }^2 \dx \\
	&\quad= -\frac{\mollEps^2}{\DE} \intOmSph \QL(\psiL{n}) \nablax \potL{n} \cdot \nablax \omegaL{n} \dH\dx \\
    &\quad\leq \frac{\mollEps^2}{2 \DE} \intOm \abs{\nablax \omegaL{n}}^2 \dx + C\lrcb{\DE^{-1}, \mollEps, \CmollEps, \potStr, d} \norm{\psi_0}_{\LpOS{1}}^2 \intOm \abs{\omegaL{n}}^2 \dx.
	\end{align*}
	
	Thus, a discrete integration in time and \eqref{init:psi_0} yield
	\begin{align}
	\label{eq:omegaL_combined}
	\begin{split}
	&\frac{1}{2} \intOm \abs{\omegaL{n}}^2 \dx + \frac{1}{2} \sum_{k=1}^{n} \intOm \abs{\omegaL{k} - \omegaL{k-1}}^2 \dx + \frac{\mollEps^2}{2 \DE} \sum_{k=1}^{n} \tau \intOm \abs{\nablax \omegaL{k}}^2 \dx \\
	&\quad\leq \frac{1}{2} \norm{\psi_0}_{\L^2(\Omega; \L^1(\Sph))}^2 + C_1\lrcb{\DE^{-1}, \mollEps, \CmollEps, \potStr, d} \norm{\psi_0}_{\LpOS{1}}^2 \sum_{k=1}^n \tau \intOm \abs{\omegaL{k}}^2 \dx
	\end{split}
	\end{align}
	for $n \in \{1, \dots, N\}$, where $C_1 = C_1\lrcb{\DE^{-1}, \mollEps, \CmollEps, \potStr, d} > 0$ is independent of $\tau$ and $\regL$.
    We complete the proof by restricting the range of the time increment to 
    \begin{align*}
    0 < \tau < \frac{1}{4 C_1 \norm{ \psi_0 }_{\LpOS{1}}^2},
    \end{align*}
    and by applying a discrete version of Gronwall's inequality (cf. \cite[Lemma 4.3.2]{werner1986ode}).
\end{proof}

\subsection{Time regularity}
\label{sec:time_regularity}
In this section, we derive bounds on the discrete time derivatives of $\uL{n}$ and $\psiL{n}$ in appropriate spaces to apply compactness results to be able to extract strongly converging subsequences.

\begin{lmm}[Time regularity of $\uflow$]
	\label{lmm:time_reg:uL}
	Suppose that the assumptions \eqref{assumptions:domain_init_data} hold and that $\tau > 0$ satisfies the condition of Lemma \ref{lmm:regularity_omegaL}.
	Then, we have that
	\begin{align}
	\label{eq:time_reg:uL}
	\sum_{k=1}^N \tau \norm{ \frac{\uL{k} - \uL{k-1}}{\tau} }_{ (\Vk{2})^\prime}^2 \leq C,
	\end{align}
	where $C>0$ depends solely on $\RE, \RE^{-1}, \DE, \DE^{-1}, \viscFrac, \viscFrac^{-1}, \activity, T$, $\CmollEps, \mollEps, \potStr, d, \Omega, \uflow_0,$ and $\psi_0$.
\end{lmm}

\begin{proof}
	Let $\wflow \in \Vk{2} \hookrightarrow \L^\infty(\Omega)$. Then, we infer from \eqref{eq:PLtau:u} and \eqref{reg:QL_leq} that, for $n \in \{1, \dots, N\}$,
	\begin{align*}
	&\RE\DE \abs{ \intOm \frac{\uL{n} - \uL{n-1}}{\tau} \cdot \wflow \dx } \\
	&\quad \leq \RE\DE \intOm \abs{\uL{n-1}} \abs{\nablax \uL{n}} \abs{\wflow} \dx + \viscFrac\DE \norm{\nablax \uL{n}}_{\L^2(\Omega)} \norm{\nablax \wflow}_{\L^2(\Omega)} \\
	&\qquad + \frac{(1-\viscFrac) \DE}{2} \intOmSph \QL(\psiL{n}) \abs{\nablax \uL{n} : \mom} \abs{\nablax \wflow : \mom} \dHm \dx \\
	&\qquad + (1-\viscFrac) \intOmSph \abs{\psiL{n}} \abs{\nablax \potL{n}} \abs{\wflow} \dH\dx  \\
	&\qquad + (1-\viscFrac) C(d) \intOmSph \abs{\psiL{n}} \abs{\nablag \potL{n}} \abs{\nablax \wflow} \dH\dx \numberthis[eq:tr:uL:tested] \\
	&\qquad + (1-\viscFrac) \lrcb{ C(d) + \activity } \intOmSph \abs{\psiL{n}} \abs{\nablax \wflow} \dH\dx \\
	&\quad =: I + II + III + IV + V + VI.
	\end{align*}
	
	First, on noting \eqref{eq:mass_consv_psiLn}, \eqref{reg:QL_leq}, \eqref{init:regularity_psi_L_0}, and \eqref{eq:regularity_omegaL}, applying H\"older's inequality yields
	\begin{align*}
	III
	&\leq C(\viscFrac, \DE, d) \lrcb{ \intOmSph \QL(\psiL{n}) \lrcb{ \nablax \uL{n} : \mom }^2 \dHm \dx }^\frac{1}{2} \cdot \\
	&\qquad\qquad\qquad \cdot \norm{\nablax \wflow}_{\L^4(\Omega)} \norm{\omegaL{n}}_{\L^2(\Omega)}^{\frac{1}{2}} \\
	&\leq C \lrcb{ \intOmSph \QL(\psiL{n}) \lrcb{ \nablax \uL{n} : \mom }^2 \dHm \dx }^\frac{1}{2} \norm{\wflow}_{\Vk{2}}, \numberthis[eq:tr:uL:III]
	\end{align*}
	where $C = C(\viscFrac, \DE, d, \DE^{-1}, \mollEps, \CmollEps, \potStr, \psi_0) >0$.
	Furthermore, we deduce from \eqref{eq:nabla_potLk_bound}, \eqref{def:omegaL}, and \eqref{eq:regularity_omegaL} that
	\begin{align}
	\label{eq:tr:uL:IV_V}
	\begin{split}
	IV + V
	&\leq C\lrcb{ \viscFrac, d, \CmollEps, \potStr } \norm{\psi_0}_{\LpOS{1}} \lrcb{ \max_{k \in \{1, \dots, N\}} \norm{\omegaL{k}}_{\L^2(\Omega)} } \norm{\wflow}_{\Vk{2}} \\
	&\leq C\lrcb{\viscFrac, d, \CmollEps, \potStr, \DE^{-1}, \mollEps, T, \psi_0} \norm{\wflow}_{\Vk{2}}.
	\end{split}
	\end{align}
    
    The other terms in \eqref{eq:tr:uL:tested} are handled easily such that combining \eqref{eq:tr:uL:tested} - \eqref{eq:tr:uL:IV_V} yields
	\begin{align}
	\label{eq:tr:uL:combined}
	\begin{split}
	&\abs{ \intOm \frac{\uL{n} - \uL{n-1}}{\tau} \cdot \wflow \dx } \\
	&\quad\leq C  \bigg( 1 + \norm{\nablax \uL{n}}_{\L^2(\Omega)} \\
	&\qquad\qquad + \lrcb{ \intOmSph \QL(\psiL{n}) \lrcb{ \nablax \uL{n} : \mom }^2 \dHm\dx }^\frac{1}{2} \bigg) \norm{\wflow}_{\Vk{2}},
	\end{split}
	\end{align}
	where $C = C\lrcb{ \RE, \RE^{-1}, \DE, \DE^{-1}, \viscFrac, \activity, T, \CmollEps, \mollEps, \potStr, d, \Omega, \uflow_0, \psi_0 } > 0$.
	Hence, on noting \eqref{est:discrete_entropy}, multiplying \eqref{eq:tr:uL:combined} by $\tau$ and summing up from $n = 1 \to N$, we complete the proof.
\end{proof}

\begin{lmm}[Time regularity of $\psi$]
	\label{lmm:time_reg:psiL}
	Suppose that the assumptions \eqref{assumptions:domain_init_data} hold and that $\tau > 0$ satisfies the condition of Lemma \ref{lmm:regularity_omegaL}.
	Then, we have that
	\begin{align}
	\label{eq:time_reg:psiL}
	\sum_{k=1}^N \tau \norm{\frac{\psiL{k} - \psiL{k-1}}{\tau}}_{ \lrcb{\HkOS{4}}^\prime }^2 \leq C,
	\end{align}
	where $C>0$ depends solely on $\Omega, \RE, \RE^{-1}, \DE, \DE^{-1}, \viscFrac, \viscFrac^{-1}, \activity, T, \CmollEps, \mollEps, \mollEps^{-1}, \potStr, d, \uflow_0$, and $\psi_0$.
\end{lmm}

\begin{proof}
	Let $n \in \{1, \dots, N\}$ and $\theta \in \H^4(\OmSph) \hookrightarrow \W^{1, \infty}(\OmSph)$. Then, we infer from \eqref{eq:PLtau:psi} that
	\begin{align*}
	&\abs{ \intOmSph \frac{\psiL{n} - \psiL{n-1}}{\tau} \theta \dH\dx } \\
	&\quad\leq \intOmSph \abs{\psiL{n}} \abs{\uL{n}} \abs{\nablax \theta} \dH\dx \\
	&\quad\quad + \frac{\mollEps^2}{\DE} \intOmSph \abs{\nablax \psiL{n}} \abs{\nablax \theta} \dH\dx
	+ \frac{1}{\DE} \intOmSph \abs{\nablag \psiL{n}} \abs{\nablag \theta} \dH\dx \\
	&\quad\quad + \frac{\mollEps^2}{\DE} \intOmSph \QL(\psiL{n}) \abs{\nablax \potL{n}} \abs{\nablax \theta} \dH\dx \\
	&\quad\quad+ \frac{1}{\DE} \intOmSph \QL(\psiL{n}) \abs{\nablag \potL{n}} \abs{\nablag \theta} \dH\dx \\
	&\quad\quad + C(d) \intOmSph \QL(\psiL{n}) \abs{\nablax \uL{n}} \abs{\nablag \theta} \dH\dx \\
	&\quad=: I + II + III + IV + V + VI. \numberthis[eq:tr:psiL:tested]
	\end{align*}
    
	First, we obtain from \eqref{reg:QL_leq}, \eqref{eq:mass_consv_psiLn}, \eqref{init:regularity_psi_L_0}, and \eqref{eq:nabla_potLk_bound} that
	\begin{align}
	\label{eq:tr:psiL:IV_V}
	\begin{split}
	IV + V
	&\leq C(\mollEps, \DE^{-1}) \Big( \norm{\nablax \potL{n}}_{\LpOS{\infty}} \\
	&\qquad\qquad + \norm{\nablag \potL{n}}_{\LpOS{\infty}} \Big) \norm{\psiL{n}}_{\LpOS{1}} \norm{\theta}_{\W^{1, \infty}(\OmSph)} \\
	&\leq C(\mollEps, \DE^{-1}, \Omega, d) \norm{\psi_0}_{\LpOS{1}}^2 \norm{\theta}_{\HkOS{4}}.
	\end{split}
	\end{align}
    The estimates of the other terms in \eqref{eq:tr:psiL:tested} can be found in \cite[p.25]{chen2017activeLC}.
	Thus, we conclude that
	\begin{align}
	\label{eq:tr:psiL:combined}
	\begin{split}
	&\abs{ \intOmSph \frac{\psiL{n} - \psiL{n-1}}{\tau} \theta \dH\dx } \\
	&\quad\leq C \Big( 1 + \norm{ \nablax \sqrt{\psiL{n}} }_{\LpOS{2}} \\
	&\qquad\qquad + \norm{ \nablag \sqrt{\psiL{n}} }_{\LpOS{2}} + \norm{\nablax \uL{n}}_{\L^2(\Omega)} \Big) \norm{\theta}_{\HkOS{4}},
	\end{split}
	\end{align}
	where $C = C\lrcb{ \Omega, \RE, \RE^{-1}, \DE, \DE^{-1}, \viscFrac, \activity, T, \CmollEps, \mollEps, \potStr, d, \uflow_0, \psi_0 } > 0$.
    Now, to complete the proof, we multiply \eqref{eq:tr:psiL:combined} by $\tau$, sum up from $n = 1 \to N$, and apply \eqref{est:discrete_entropy}.
\end{proof}

\section{Passage to the limit: Existence of weak solutions}
\label{sec:passage_limit}
In this section, our goal is to pass to the limit in \eqref{def:PLtau} as $\regL \nearrow \infty$ and $\tau \searrow 0$. 
Before doing so however we reformulate \eqref{def:PLtau} continuously in time. 
For this task, let us define the following continuous-in-time functions:
\begin{equation}
\label{def:time_cont_fcts}
\begin{alignedat}{2}
\uL{\taup}(\cdot, t) &:= \uL{n}(\cdot), \qquad \uL{\taum}(\cdot, t) := \uL{n-1}(\cdot),\quad && t \in (t_{n-1}, t_n] , \quad n \in \{1, \dots, N\}, \\
\uL{\taum}(\cdot, 0) &:= \uL{0}(\cdot) \in \Hdiv.
\end{alignedat}
\end{equation}

Moreover, we define similarly to Chen and Liu (cf. \cite[p. 26]{chen_liu_2013_entropy_solution}) the difference quotient with respect to time by
\begin{align}
\label{def:time_deriv}
\dtau \uL{\taup}(\cdot, t) := \frac{\uL{n}(\cdot) - \uL{n-1}(\cdot)}{\tau}, \quad t \in (t_{n-1}, t_n], \quad n \in \{1, \dots, N\}.
\end{align}
We introduce the same notation for $\psiL{\tau, \pm}$ and $\dtau \psiL{\taup}$
with $\psiL{\taum}(\cdot, \tilde{\cdot}, 0) = \psiL{0}(\cdot, \tilde{\cdot})$.
From now on, let
\begin{align}
\label{tc:cond:tau}
\tau \in o(\regL^{-1}) \text{ satisfy the condition in Lemma } \ref{lmm:regularity_omegaL}.
\end{align}

On noting the time continuous notation \eqref{def:time_cont_fcts}, we can rewrite the weak formulation \eqref{def:PLtau} after a discrete integration with respect to time as follows:\\
Find $\uL{\tau, \pm}(\cdot, t) \in \Hdiv$, $t \in (0,T]$, such that
\begin{align*}
&\RE \DE \intT\intOm \dtau \uL{\taup} \cdot \wflow \dx\dt + \RE \DE \intT\intOm \lrcb{ (\uL{\taum} \cdot \nablax) \uL{\taup} } \cdot \wflow \dx\dt \\
&\quad\quad + \viscFrac \DE \intT\intOm \nablax \uL{\taup} : \nablax \wflow \dx\dt \\
&\quad\quad + \frac{(1-\viscFrac) \DE}{2} \intT\intOmSph \QL(\psiL{\taup}) \lrcb{ \nablax \uL{\taup} : \mom } \cdot \\
&\qquad\qquad\qquad\qquad\qquad\qquad\qquad\qquad \cdot \lrcb{\nablax \wflow : \mom} \dHm\dx\dt \\
&\quad= - \frac{(1-\viscFrac)}{2} \intT\intOmSph \psiL{\taup} \nablax \pot[\psiL{\taup} + \psiL{\taum}] \cdot \wflow \dH\dx\dt \numberthis[tc:eq:u] \\
&\quad\quad - \frac{1-\viscFrac}{2} \intT\intOmSph \QL(\psiL{\taup}) \lrcb{ (\IdMat - \mom) \nablax \wflow \mflow } \cdot \\
&\qquad\qquad\qquad\qquad\qquad\qquad\qquad\qquad \cdot \nablag \pot[\psiL{\taup} + \psiL{\taum}] \dHm\dx\dt \\
&\quad\quad - (1-\viscFrac) \intT\intOmSph \psiL{\taup} \lrcb{d\mom - \IdMat} : \nablax \wflow \dHm\dx\dt \\
&\quad\quad - \activity (1 - \viscFrac) \intT\intOmSph \QL(\psiL{\taup}) \lrcb{ \mom : \nablax \wflow } \dHm\dx\dt
\end{align*}
for all $\wflow \in \L^1(0,T; \Hdiv)$.
Moreover, find $\psiL{\tau, \pm}(\cdot, \tilde{\cdot}, t) \in \HkOS{1} \cap \Zp{2}$, $t \in (0,T]$, such that
\begin{align*}
&\intT\intOmSph \dtau \psiL{\taup} \theta \dH\dx\dt - \intT\intOmSph \psiL{\taup} \uL{\taup} \cdot \nablax \theta \dH\dx\dt \\
&\qquad + \frac{1}{\DE} \intT\intOmSph \left[ \mollEps^2 \nablax \psiL{\taup} \cdot \nablax \theta + \nablag \psiL{\taup} \cdot \nablag \theta \right] \dH\dx\dt \\
&\qquad + \frac{\mollEps^2}{2 \DE} \intT\intOmSph \QL(\psiL{\taup}) \nablax \pot[\psiL{\taup} + \psiL{\taum}] \cdot \nablax \theta \dH\dx\dt \numberthis[tc:eq:psi] \\
&\qquad + \frac{1}{2 \DE} \intT\intOmSph \QL(\psiL{\taup}) \nablag \pot[\psiL{\taup} + \psiL{\taum}] \cdot \nablag \theta \dH\dx\dt \\
&\qquad - \intT\intOmSph \QL(\psiL{\taup}) \lrcb{ (\IdMat - \mom) \nablax \uL{\taup} \mflow } \cdot \nablag \theta \dHm\dx\dt = 0
\end{align*}
for all $\theta \in \L^1(0,T; \HkOS{1})$.
We note that the problem \eqref{tc:eq:u}-\eqref{tc:eq:psi} is equivalent to \eqref{eq:PLtau:u}-\eqref{eq:PLtau:psi}, to which existence of solutions has been proved (cf. Lemma \ref{lmm:existence_fp}).
Furthermore, let us define for future purposes
\begin{subequations}
	\begin{align}
	\omegaL{\tau, \pm}(\xflow, t) &:= \intSph \psiL{\tau, \pm}(\xflow, \mflow, t) \dHm \in \R_{\geq 0}, \label{tc:def:omegaL} \\
	\StressL{\tau, \pm}(\xflow, t) &:= \intSph \psiL{\tau, \pm}(\xflow, \mflow, t) \mom \dHm \in \R^{d \times d} \label{tc:def:StressL}
	\end{align}
\end{subequations}
for a.e. $(\xflow, t) \in \Omega \times (0, T]$. Thus, we can rewrite the interaction potential as follows:
\begin{align}
\label{tc:pot_repres}
\pot[\psiL{\tau, \pm}](\xflow, \mflow, t) = \potStr \Je{\omegaL{\tau, \pm}(\cdot, t)}(\xflow) - \potStr \JeMat{\StressL{\tau, \pm}(\cdot, t)}(\xflow) : \mom
\end{align}
for $(\xflow, \mflow, t) \in \OmSph \times (0, T]$.
Moreover, we obtain from \eqref{init:regularity_u_L_0}, $\tau \in o(\regL^{-1})$, \eqref{def:time_cont_fcts}, \eqref{est:discrete_entropy}, and \eqref{eq:regularity_omegaL} that
\begin{align*}
&\esssup_{t \in [0,T]} \intOm \abs{\uL{\tau, \pm}(\xflow, t)}^2 \dx  +  \frac{1}{\tau} \intT\intOm \abs{ \uL{\taup} - \uL{\taum} }^2 \dx\dt  \\
&\quad\quad+  \intT\intOm \abs{\nablax \uL{\tau, \pm}}^2 \dx\dt \\
&\quad\quad + \esssup_{t \in [0,T]} \intOmSph \entrop(\psiL{\taup}(\xflow, \mflow, t)) \dHm\dx \\
&\quad\quad + \frac{1}{\tau \regL} \intT\intOmSph \abs{ \psiL{\taup} - \psiL{\taum} }^2 \dH\dx\dt \numberthis[tc:regularity_past] \\
&\quad\quad + \intT\intOmSph \Big[ | \nablax \sqrt{\psiL{\taup}} |^2 + | \nablag \sqrt{\psiL{\taup}} |^2 \Big] \dH\dx\dt \\
&\quad\quad + \esssup_{t \in [0,T]} \intOm \abs{\omegaL{\taup}(\xflow, t)}^2 \dx + \intT\intOm \abs{\nablax \omegaL{\taup}}^2 \dx\dt \\
&\quad\quad + \frac{1}{\tau} \intT\intOm \abs{ \omegaL{\taup} - \omegaL{\taum} }^2 \dx\dt \\
&\quad\leq C_\star < \infty,
\end{align*}
where $C_\star = C_\star \lrcb{ \RE, \RE^{-1}, \DE, \DE^{-1}, \viscFrac, \viscFrac^{-1}, \activity, T, \CmollEps, \mollEps, \mollEps^{-1}, \potStr, d, \Omega, \uflow_0, \psi_0 } > 0$ is independent of $\tau$ and $\regL$.

\begin{thrm}[Main Theorem: Existence of weak solutions]
	\label{thrm:main}
	Suppose that the assumptions \eqref{assumptions:domain_init_data} and the condition \eqref{tc:cond:tau}, which relates $\tau$ to $\regL$, hold.
	Then, there exists a subsequence of $\{ (\uL{\tau, \pm}, \psiL{\tau, \pm}) \}_{\regL > 1}$ (not relabelled) and a pair of functions $(\uflow, \psi)$ such that
	\begin{subequations}
		\begin{gather}
		\uflow \in \L^\infty(0, T; \Lflow^2(\Omega)) \cap \L^2(0, T; \Hdiv), \label{tc:limit:u} \\
		\partial_t \uflow \in \L^2\lrcb{ 0,T; ( \Vk{2} )^\prime }, \label{tc:limit:dt_u}
		\end{gather}
	\end{subequations}
	and
	\begin{subequations}
		\begin{gather}
		\psi \in \L^\infty(0,T; \L^2(\Omega; \L^1(\Sph))), \quad \sqrt{\psi} \in \L^2(0,T; \HkOS{1}), \label{tc:limit:psi} \\
		\partial_t \psi \in \L^2( 0,T; (\HkOS{4})^\prime ), \label{tc:limit:dt_psi}
		\end{gather}
	\end{subequations}
	with $\psi \geq 0$ a.e. in $\OmSph \times [0,T]$ and
	\begin{align}
	\label{tc:limit:finite_entropy}
	\entrop(\psi) \in \L^\infty(0, T; \LpOS{1}).
	\end{align}
	Moreover, it holds that, as $\regL \nearrow \infty$ (and thereby $\tau \searrow 0$),
	\begin{subequations}
		\begin{alignat}{3}
		\uL{\tau, \pm} &\weakstar \uflow \quad &&\text{ weak}^* &&\text{ in } \L^\infty(0, T; \Lflow^2(\Omega)) \label{tc:limit:u:weakstar} \\
		\uL{\tau, \pm} &\weak \uflow \quad &&\text{ weakly} &&\text{ in } \L^2(0,T; \Hdiv), \label{tc:limit:u:weak} \\
		\uL{\tau, \pm} &\to \uflow \quad &&\text{ strongly} &&\text{ in } \L^2(0,T; \Lflow^r(\Omega)), \label{tc:limit:u:strong}\\
		\dtau \uL{\taup} &\weak \partial_t \uflow \quad &&\text{ weakly} &&\text{ in } \L^2( 0,T; (\Vk{2})^\prime ), \label{tc:limit:u:dt}
		\end{alignat}
	\end{subequations}
	where $r \in [1, \infty)$ if $d = 2$ and $r \in [1, 6)$ if $d = 3$, and
	\begin{subequations}
		\begin{alignat}{3}
		\nablax \sqrt{\psiL{\taup}} &\weak \nablax \sqrt{\psi} \quad &&\text{ weakly} &&\text{ in } \L^2(0,T; \Lflow^2(\OmSph)), \label{tc:limit:psi:nablax_sqrt} \\
		\nablag \sqrt{\psiL{\taup}} &\weak \nablag \sqrt{\psi} \quad &&\text{ weakly} &&\text{ in } \L^2(0,T; \Lflow^2(\OmSph)), \label{tc:limit:psi:nablag_sqrt} \\
		\sqrt{\psiL{\taup}} &\to \sqrt{\psi} \quad&&\text{ strongly} &&\text{ in } \L^4(0,T; \L^4(\Omega; \L^2(\Sph))), \label{tc:limit:psi:sqrt_strong} \\
		\psiL{\tau, \pm} &\to \psi \quad&&\text{ strongly} &&\text{ in } \L^2(0,T; \L^2(\Omega; \L^1(\Sph))), \label{tc:limit:psi:strong} \\
		\QL(\psiL{\taup}) &\to \psi \quad&&\text{ strongly} &&\text{ in } \L^2(0,T; \L^2(\Omega; \L^1(\Sph))), \label{tc:limit:psi:QL_strong} \\
		\dtau \psiL{\taup} &\weak \partial_t \psi \quad&&\text{ weakly} &&\text{ in } \L^2( 0,T; (\HkOS{4})^\prime ). \label{tc:limit:psi:dt}
		\end{alignat}
	\end{subequations}
	Furthermore, we introduce
	\begin{subequations}
		\label{tc:def:omega_Stress}
		\begin{align}
		\omega(\xflow, t) &:= \intSph \psi(\xflow, \mflow, t) \dHm \in \R_{\geq 0} \ \text{ and } \label{tc:def:omega}\\
		\Stress(\xflow, t) &:= \intSph \psi(\xflow, \mflow, t) \mom \dHm \in \R^{d \times d} \label{tc:def:Stress}
		\end{align}
	\end{subequations}
	for a.e. $(\xflow, t) \in \Omega \times (0,T]$. It holds that
	\begin{subequations}
		\begin{align}
		\omega &\in \L^\infty(0, T; \L^2(\Omega)) \cap \L^2(0, T; \H^1(\Omega)), \label{tc:limit:omega} \\
		\Stress &\in \L^\infty(0, T; \LMat^2(\Omega)), \label{tc:limit:Stress}
		\end{align}
	\end{subequations}
	and that
	\begin{subequations}
		\begin{alignat}{3}
		\omegaL{\tau, \pm} &\to \omega \quad&&\text{ strongly} &&\text{ in } \L^2(0,T; \L^2(\Omega)), \label{tc:limit:omega:strong} \\
		\StressL{\tau, \pm} &\to \Stress \quad&&\text{ strongly} &&\text{ in } \L^2(0,T; \LMat^2(\Omega)), \label{tc:limit:Stress:strong} \\
		\pot[\psiL{\tau, \pm}] &\to \pot[\psi] \quad&&\text{ strongly} &&\text{ in } \L^2(0, T; \W^{1, \infty}(\OmSph)). \label{tc:limit:pot:strong}
		\end{alignat}
	\end{subequations}
	
	Moreover, the tuple $(\uflow, \psi)$ is a global weak solution of $(\mathcal{P})$, in the sense that
	\begin{align}
	\label{P:weak:u}
	\begin{split}
	&-\RE\DE \intT\intOm \uflow \cdot \partial_t \wflow \dx\dt  +  \RE\DE \intT\intOm \lrcb{ (\uflow \cdot \nablax) \uflow } \cdot \wflow \dx\dt \\
	&\qquad + \viscFrac \DE \intT\intOm \nablax \uflow : \nablax \wflow \dx\dt \\
	&\qquad + \frac{(1-\viscFrac) \DE}{2} \intT\intOmSph \psi (\nablax \uflow : \mom) (\nablax \wflow : \mom) \dHm\dx\dt \\
	&\quad = \RE\DE \intOm \uflow_0(\xflow) \cdot \wflow(\xflow, 0) \dx  -  (1 - \viscFrac) \intT\intOmSph \psi \nablax \pot[\psi] \cdot \wflow \dH\dx\dt \\
	&\quad\qquad - (1 - \viscFrac) \intT\intOmSph \psi \lrcb{ (\IdMat - \mom) \nablax \wflow \mflow } \cdot \nablag \pot[\psi] \dHm\dx\dt \\
	&\quad\qquad - (1 - \viscFrac) \intT\intOmSph \psi (d \mom - \IdMat) : \nablax \wflow \dHm\dx\dt \\
	&\quad\qquad - \alpha (1 - \viscFrac) \intT\intOmSph \psi (\mom : \nablax \wflow) \dHm\dx\dt \\
	&\qquad\qquad\qquad \forall \wflow \in \W^{1,1}(0,T; \Vk{2}) \text{ such that } \wflow(\cdot, T) = 0,
	\end{split}
	\end{align}
	and
	\begin{align*}
	&-\intT\intOmSph \psi \partial_t \theta \dH\dx\dt  -  \intT\intOmSph \psi \uflow \cdot \nablax \theta \dH\dx\dt \\
	&\qquad + \frac{1}{\DE} \intT\intOmSph \left[ \mollEps^2 \nablax \psi \cdot \nablax \theta + \nablag \psi \cdot \nablag \theta \right] \dH\dx\dt \\
	&\qquad + \frac{1}{\DE} \intT\intOmSph \psi \left[ \mollEps^2 \nablax \pot[\psi] \cdot \nablax \theta + \nablag \pot[\psi] \cdot \nablag \theta \right] \dH\dx\dt \\
	&\qquad - \intT\intOmSph \psi \lrcb{ (\IdMat - \mom) \nablax \uflow \mflow } \cdot \nablag \theta \dHm\dx\dt \\
	&\quad = \intOmSph \psi_0(\xflow, \mflow) \theta(\xflow, \mflow, 0) \dHm\dx \numberthis[P:weak:psi] \\
	&\qquad\qquad\qquad \forall \theta \in \W^{1,1}(0,T; \HkOS{4}) \text{ such that } \theta(\cdot, \tilde{\cdot}, T) = 0.
	\end{align*}
	
	In particular, the following energy inequality is satisfied for a.e. $s \in [0,T]$:
	\begin{align*}
	&\frac{\RE\DE}{2} \intOm \abs{ \uflow(\xflow, s) }^2 \dx  +  \frac{\viscFrac\DE}{2} \int_0^s \abs{ \nablax \uflow }^2 \dx\dt \\
	&\qquad + (1 - \viscFrac) \intOmSph \entrop(\psi(\xflow, \mflow, s)) \dHm\dx \\
	&\qquad +  \frac{1-\viscFrac}{2} \intOmSph \psi(\xflow, \mflow, s) \pot[\psi(\xflow, \mflow, s)] \dHm\dx \\
	&\qquad + \frac{1 - \viscFrac}{\DE} \int_0^s \intOmSph \left[ \mollEps^2 \abs{ \nablax \sqrt{\psi} }^2 + \abs{\nablag \sqrt{\psi}}^2 \right] \dH\dx\dt \\
	&\quad \leq \frac{\RE\DE}{2} \intOm \abs{\uflow_0}^2 \dx + (1 - \viscFrac) \intOmSph \entrop(\psi_0) \dH\dx \numberthis[eq:main:energy_ineq] \\
	&\quad\quad + \frac{1-\viscFrac}{2} \intOmSph \psi_0 \pot[\psi_0] \dH\dx + C_{\star\star},
	\end{align*}
	where $C_{\star\star} > 0$ depends solely on $\psi_0, \alpha, \RE, \DE, \DE^{-1}, \viscFrac, \CmollEps, \potStr, \Omega, d,$ and  $T$.

\end{thrm}

\begin{proof}
    For ease of notation, we note that all generic constants $0 < C < \infty$ are independent of $\regL$ and $\tau$ in this proof.\\
	\textit{Convergence properties of $\uflow$.}
	First, we prove the convergence properties of the velocity field. We infer from the first bound in \eqref{tc:regularity_past} that there exist subsequences $\{ \uL{\tau, \pm} \}_{\regL > 1}$ (not relabelled) and functions $\uflow^{\pm} \in \L^\infty(0,T; \Lflow^2(\Omega))$ such that
	\begin{align}
	\label{tc:limit:proof:u:weakstar}
	\uL{\tau, \pm} \weakstar \uflow^{\pm} \quad \text{ weak}^* \text{ in } \L^\infty(0,T; \Lflow^2(\Omega)).
	\end{align}
	Moreover, on noting \eqref{tc:limit:proof:u:weakstar} and $\tau \in o(\regL^{-1})$, the second bound in \eqref{tc:regularity_past} yields
	\begin{align*}
	\intT\intOm \lrcb{ \uflow^+ - \uflow^- } \cdot \wflow \dx\dt
	\gets \intT\intOm \lrcb{ \uL{\taup} - \uL{\taum} } \cdot \wflow \dx\dt 
	\to 0
	\end{align*}
	for all $\wflow \in \L^2(0,T; \Lflow^2(\Omega)) \subset \L^1(0,T; \Lflow^2(\Omega))$. Hence, we have that $\uflow^+ = \uflow^-$ and we denote this limit by $\uflow := \uflow^+ = \uflow^- \in \L^\infty(0,T; \Lflow^2(\Omega))$, which yields \eqref{tc:limit:u:weakstar}.
	Furthermore, we obtain from the uniqueness of limits of sequences in the weak topology of $\L^2(0,T; \Hdiv)$, the third bound in \eqref{tc:regularity_past}, and a further extraction of a subsequence that $\uflow \in \L^2(0,T; \Hdiv)$ and that \eqref{tc:limit:u} and \eqref{tc:limit:u:weak} hold.
	On noting \eqref{eq:time_reg:uL}, we have that
	\begin{align}
	\label{tc:limit:proof:u:timereg}
	\begin{split}
	\int_0^{T - \tau} \norm{ \uL{\taup}(\cdot, t + \tau) - \uL{\taup}(\cdot, t) }_{( \Vk{2} )^\prime}^2 dt = \sum_{k=1}^{N-1} \tau \norm{ \uL{k} - \uL{k-1} }_{( \Vk{2} )^\prime}^2
	\leq C \tau^2.
	\end{split}
	\end{align}
	On noting \eqref{tc:limit:proof:u:timereg}, $\norm{\uL{\taup}}_{\L^2(0,T; \Hdiv)} \leq C$, the compact embedding 
	\begin{align*}
	\Hdiv \hookrightarrow\hookrightarrow \Lflow^r(\Omega)\cap\Ldivp{2}	
	\end{align*}
	with $r \in [2, \infty)$ if $d=2$ and $r \in [2, 6)$ if $d=3$, and the continuous embedding
	\begin{align}
	\label{tc:limit:proof:cont_emb:u}
	\Lflow^r(\Omega) \cap \Ldivp{2} \hookrightarrow (\Vk{2})^\prime,
	\end{align}
	it follows by Theorem 1 of Dreher and J\"ungel \cite{dreher2012compact} (for a further subsequence) that $\uL{\taup} \to \uflow$ strongly in $\L^2(0,T; \Lflow^r(\Omega))$.
	As $\{ \uL{\taup} \}_{\regL > 1}$ is uniformly bounded in $\L^2(0,T;$ $\Hdiv)$, we deduce from the second bound in \eqref{tc:regularity_past} and the continuous embedding \eqref{tc:limit:proof:cont_emb:u} that $\uL{\taum} \to \uflow$ strongly in $\L^2(0,T; \Lflow^r(\Omega))$, which yields \eqref{tc:limit:u:strong}.
    The convergence result \eqref{tc:limit:u:dt} with respect to the time derivative follows similarly to \cite[p. 29ff]{chen_liu_2013_entropy_solution}.\\
	\textit{Convergence properties of $\psi$.}
	We obtain from \eqref{eq:time_reg:psiL} that
	\begin{align}
	\label{tc:limit:proof:psi:timereg}
	\begin{split}
	&\int_0^{T - \tau} \norm{ \psiL{\taup}(\cdot, \tilde{\cdot}, t + \tau) - \psiL{\taup}(\cdot, \tilde{\cdot}, t) }_{(\HkOS{4})^\prime}^2 \dt \\
	&\qquad\qquad\qquad\quad = \sum_{k=1}^{N-1} \tau \norm{ \psiL{k} - \psiL{k-1} }_{(\HkOS{4})^\prime}^2
	\leq C \tau^2.
	\end{split}
	\end{align}
    Now, on noting \eqref{tc:limit:proof:psi:timereg}, \eqref{tc:regularity_past}, and the embeddings
    \begin{align*}
    \W^{1,1}(\OmSph) \hookrightarrow\hookrightarrow \L^s(\OmSph) \hookrightarrow \lrcb{\HkOS{4}}^\prime \quad \forall 1 < s < \frac{2d - 1}{2d -2},
    \end{align*}
    following the lines pp. 27-30 in \cite{chen_liu_2013_entropy_solution} yields the convergence results \eqref{tc:limit:psi:nablax_sqrt} - \eqref{tc:limit:psi:dt} with respect to $\{ \psiL{\taup} \}_{\regL> 1}$.
    Furthermore, we deduce from the fifth bound in \eqref{tc:regularity_past} and $\tau \in o(\regL^{-1})$ that \eqref{tc:limit:psi:strong} holds for $\{ \psiL{\taum} \}_{\regL > 1}$ as well. The nonnegativity of $\psi$ follows from the nonnegativity of $\psiL{\taup}$.
	Moreover, on noting \eqref{tc:def:omegaL}, the seventh bound in \eqref{tc:regularity_past}, and \eqref{tc:limit:psi:strong}, we have that $\psi$ $\in$ $\L^\infty(0, T; \L^2(\Omega; \L^1(\Sph)))$.
	This completes the proof of \eqref{tc:limit:psi}.
	Moreover, on noting \eqref{tc:limit:psi:strong}, it follows that $\eqref{tc:limit:finite_entropy}$ holds by subtracting one more subsequence.\\
	\textit{Convergence properties of $\omega$, $\Stress$, and $\pot[\psi]$.}
    On noting \eqref{tc:regularity_past}, \eqref{tc:def:omegaL}, \eqref{tc:def:omega}, \eqref{tc:def:StressL}, \eqref{tc:def:Stress}, \eqref{tc:pot_repres}, \eqref{mollifier:W_2_infty}, we infer from \eqref{tc:limit:psi:strong} that \eqref{tc:limit:omega} - \eqref{tc:limit:pot:strong} hold.\\
	\textit{Existence of weak solutions.}
    By applying the established convergence results in Theorem \ref{thrm:main}, it follows immediately that $\uflow$ and $\psi$ satisfy the weak formulations \eqref{P:weak:u} and \eqref{P:weak:psi}.\\
	\textit{Energy inequality.}
    On noting the convergence properties \eqref{tc:limit:u:weakstar}, \eqref{tc:limit:u:weak}, \eqref{tc:limit:u:strong}, \eqref{tc:limit:psi:nablax_sqrt}, \eqref{tc:limit:psi:nablag_sqrt}, and \eqref{tc:limit:psi:strong}, we infer from the (weak) lower-semicontinuity of the terms on the left-hand side of \eqref{est:discrete_entropy} that the energy inequality \eqref{eq:main:energy_ineq} holds by subtracting one more subsequence.
\end{proof}

\textbf{Acknowledgments.} This work has been supported by the Research Training Group (Graduiertenkolleg) 2339 “Interfaces, Complex Structures,
and Singular Limits in Continuum Mechanics” of the German Science Foundation (DFG). The author is grateful to
G\"unther Gr\"un and Stefan Metzger for many valuable discussions.

\bibliographystyle{acm}
\bibliography{bibliography}

\end{document}